\documentclass[11pt]{article}      %
\usepackage{amsmath,amssymb, amsthm, eucal,bm,accents,mathtools}  
\usepackage{graphicx, amscd}
\usepackage{framed}
\usepackage{url}               

 \usepackage{relsize} 
 \usepackage{xfrac}  

\usepackage{epstopdf}

\usepackage{txfonts}                  
\usepackage[T1]{fontenc}     

\usepackage{scrextend}     

\usepackage{enumerate}       
\usepackage{float}                 
\usepackage{tikz}
\usetikzlibrary{matrix, arrows, calc}
\usetikzlibrary{arrows,shapes,positioning}

\definecolor{gold}{rgb}{1,.70,.0}   

\theoremstyle{plain}
\newtheorem{theorem}{Theorem}[section]            
\newtheorem{proposition}[theorem]{Proposition}  

\theoremstyle{definition}
\newtheorem{definition}[theorem]{Definition}

\numberwithin{theorem}{section}
\numberwithin{equation}{section}
\numberwithin{figure}{section}
\newcommand{\gaction}[2]{\genfrac{}{}{0.5pt}{}{#1}{#2}%
                        \!\lower2pt\hbox{\rotatebox[origin=c]{-90}{{$\looparrowright$}}}}
\newcommand{\dotfraction}[2]{\genfrac{}{}{0.5pt}{}{#1}{#2}%
                        \!\lower.5pt\hbox{{$\circ$}}}

\usepackage{titlesec}                                                           
\titlelabel{\thetitle.\ }                                                         
\titleformat*{\section}{\fontsize{14pt}{14pt} \bf}        
\usepackage{fancyhdr}
\usepackage{sidecap}  
\usepackage[hang,small,sc]{caption}

\def\SL{\hbox{SL}}

\def\-{\hbox{\raisebox{.75pt}{-}}}

\usepackage{fancybox}

\def\id{\hbox{\rm id}\,}

\def\AAA{\circ}

\def\gen{\hbox{\rm gen}\,}
\def\depth{\delta}
\def\tspin{\hbox{\rm spin}}
\def\str{{+}}

\begin{document}

\title{Apollonian depth, spinors, \\
and the super-Dedekind tessellation}

\author{Jerzy Kocik 
    \\ \small Department of Mathematics, Southern Illinois University, Carbondale, IL62901
   \\ \small jkocik{@}siu.edu
}


\date{\small\today}

\maketitle

\begin{abstract}
\noindent
The configuration space of tricycles
(triples of disks in contact) is shown to coincide with the complex plane
resulting as a projective space costructed from the tangency and Pauli spinors.  
Remarkably, the fractal of the depth functions
assumes a particularly simple and  elegant form.
Moreover, the factor space due to a certain symmetry group
provides a parametrization of the Apollonian disk packings.
\\[5pt]
\\[5pt]
\end{abstract}

\section{Introduction}

In the present paper we employ the concept of tangency spinors
to parametrize the space of tricycles. 
Quite surprisingly, it brings the depth fractal to a very regular form, that  of the Apollonian belt.
Additionally,  in the same picture we obtain visualization of the classification of the Apollonian disk packing.
The implied symmetry groups aextend the modular group and lead to a  ``super-Dedekind'' tessellation,
which extends the standard modular and Dedekind tessellations.\\[-12pt]

We start with recalling a few concepts.
The Apollonian depth function, introduced in \cite{jk-web},
is defined as a function $\delta:\mathbb R^3 \to \mathbb N$
(with possible values $0$ and $\infty$)
as follows:
For a given triple of numbers $(a,b,c)$, 
if any of them is non-positive, the value of $\delta$ is 0.
Otherwise, one performs a {\bf process} in $\mathbb R^3$,
one step of which consists of replacing the greatest value in the triple by
\begin{equation}
\label{eq:abc}
a+b+c-2\sqrt{ab+bc+ca} \,.
\end{equation}
Repeat this step until the new number turns negative or $0$.
The number of steps needed to achieve it defines the value of $\delta(a,b,c)$.
Here is an example for $(179,62,23)$:
$$
(179,62,23) \ \xrightarrow{~\pi~} \ (62,23,6) \ \xrightarrow{~\pi~} \ (23,6,3) \ \xrightarrow{~\pi~} \ (3,2,-1)\quad (\hbox{terminate})
$$
Thus $\delta (179,62,23) = 3$ .

\noindent
{\bf The geometric interpretation:}
Any triple of mutually tangent disks (called further a {\bf tricycle}) 
may be completed to an Apollonian disk packing.
The depth function is a measure how ``deeply'' is
the given tricycle of curvatures $(a,b,c)$  buried in this packing.
Every step  of the process described above
corresponds to replacing the smallest 
disk by the greater from the two tangent to the 
triple $(a,b,c)$.
It is to be run until the external disk of negative (or zero) curvature is reached.

Recall, that the curvatures of four mutually tangent disks 
satisfy the Descartes formula
\cite{Descartes,Soddy,Boyd,LMW,jk-Descartes,jk-Descartes2}:
\begin{equation}
\label{eq:Descartes}
(a+b+c+d)^2 = 2\,(a^2+b^2+c^2+d^2)
\end{equation}
Its quadratic nature leads in general to two solutions:
\begin{equation}
\label{eq:Descartespm}
d = a+b+c \pm 2\sqrt{ab+bc+ca}  \,,
\end{equation}
which correspond to the two different disks 
that complete a given tricycle to a Descartes configuration, 
as illustrated in Figure \ref{fig:twosolutions}.
We choose the greater disk (smaller curvature)
in defining the process while discarding the smallest disk (greatest curvature) 
from the original triple.
This transformation will be called a {\bf Descartes ascending move.}
In general, {\bf Descartes move} is as above except the choice of the 
new disk and the one to be discarted is arvbitrary.

\begin{figure}[h]
\centering
\begin{tikzpicture}[scale=1.7]
\clip (-.79, 1.1) rectangle (1.1, -.51);
\draw [white, fill=red!60] (-.79,-2) rectangle (2,2);
\draw [very thick,  fill=white] (0,0) circle (1);
\draw [very thick,  fill=red!60] (8/23,12/23) circle (1/23);
\draw [thick, fill=yellow!20] (0/3, 2/3) circle (1/3);
\draw [thick, fill=yellow!20] (3/6, 4/6) circle (1/6);
\draw [thick, fill=yellow!20] (1/2, 0/2) circle (1/2);

\node at (3/6, 4/6) [scale=.9, color=black] {\sf a};
\node at (0/3, 2/3) [scale=.9, color=black] {\sf b};
\node at (1/2, 0/2) [scale=.9, color=black] {\sf c};
\end{tikzpicture}
\quad\qquad\qquad
\begin{tikzpicture}[scale=3.5, rotate=0]
\clip (-.4,1.05) rectangle (1.01, 0.2);

\draw [very thick, fill=red!60] (15/50, 24/50) circle (1/50);
\draw [very thick,  fill=red!60 ] (3/6, 4/6) circle (1/6);
\draw [thick, fill=yellow!20] (8/23,12/23) circle (1/23);
\draw [thick, fill=yellow!20] (0/3, 2/3) circle (1/3);
\draw [thick, fill=yellow!20] (1/2, 0/2) circle (1/2);

\node at (8/23,12/23) [scale=.9, color=black] {\sf a};
\node at (0/3, 2/3) [scale=.9, color=black] {\sf b};
\node at (1/2, 1/3) [scale=.9, color=black] {\sf c};
\end{tikzpicture}

\caption{Examples of solutions (shown as darker disks) to Descartes' problem 
for disks $a$, $b$, and $c$ .
One of the solutions on the left side has negative curvature.}
\label{fig:twosolutions}
\end{figure}
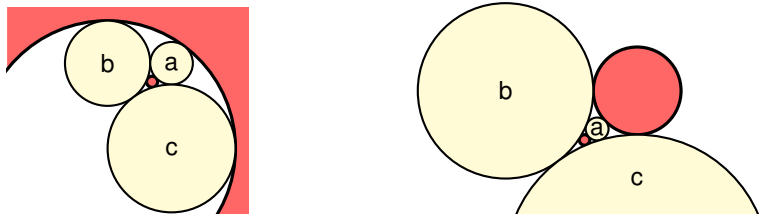


\noindent
%
%

The Apollonian depth is invariant under similarity transformations
of the tricycles, i.e., under rotations, translations and dilations.
In particular:
$$
\depth(a,b,c) = \depth(\lambda a, \lambda b , \lambda c) \qquad \lambda>0
$$
Hence the space of tricycles may be 
parametrized by two numbers,
for instance, by scaling out the tricycles  
by their greatest curvature: 
\begin{equation}
\label{eq:reduce}
(a,b,c) \ \mapsto \ (x,y) \ = \  \left(\tfrac{a}{c},\tfrac{b}{c}\right)
\end{equation}
where we assumed that $c=\max(a,b,c)$. 
The configuration space of the non-negative triples coincides with the 
unit square $I^2$.
%
%
%
%
%
%
%
%
%
%
For economy, we shall use the same symbol for this reduced depth function 
$$
\depth (x,y) \ = \  \depth(1, x, y)
$$
Figure \ref{fig:web} from \cite{jk-web} shows the plot of $\delta$ obtained 
with a computer expriment.
The degree of shade represents the value of the depth.
The intriguing fractal-like pattern 
resembles in parts  that of the Apollonian disk packing,
except the disks are replaced by ellipses.
%
%
\begin{figure}[H]
\centering
\includegraphics[scale=.5]{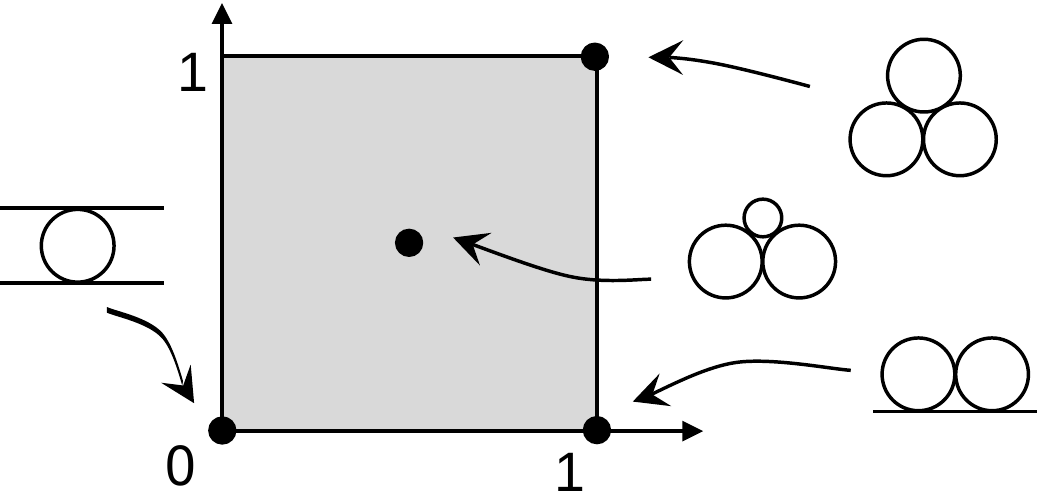}
\quad
\includegraphics[scale=.14]{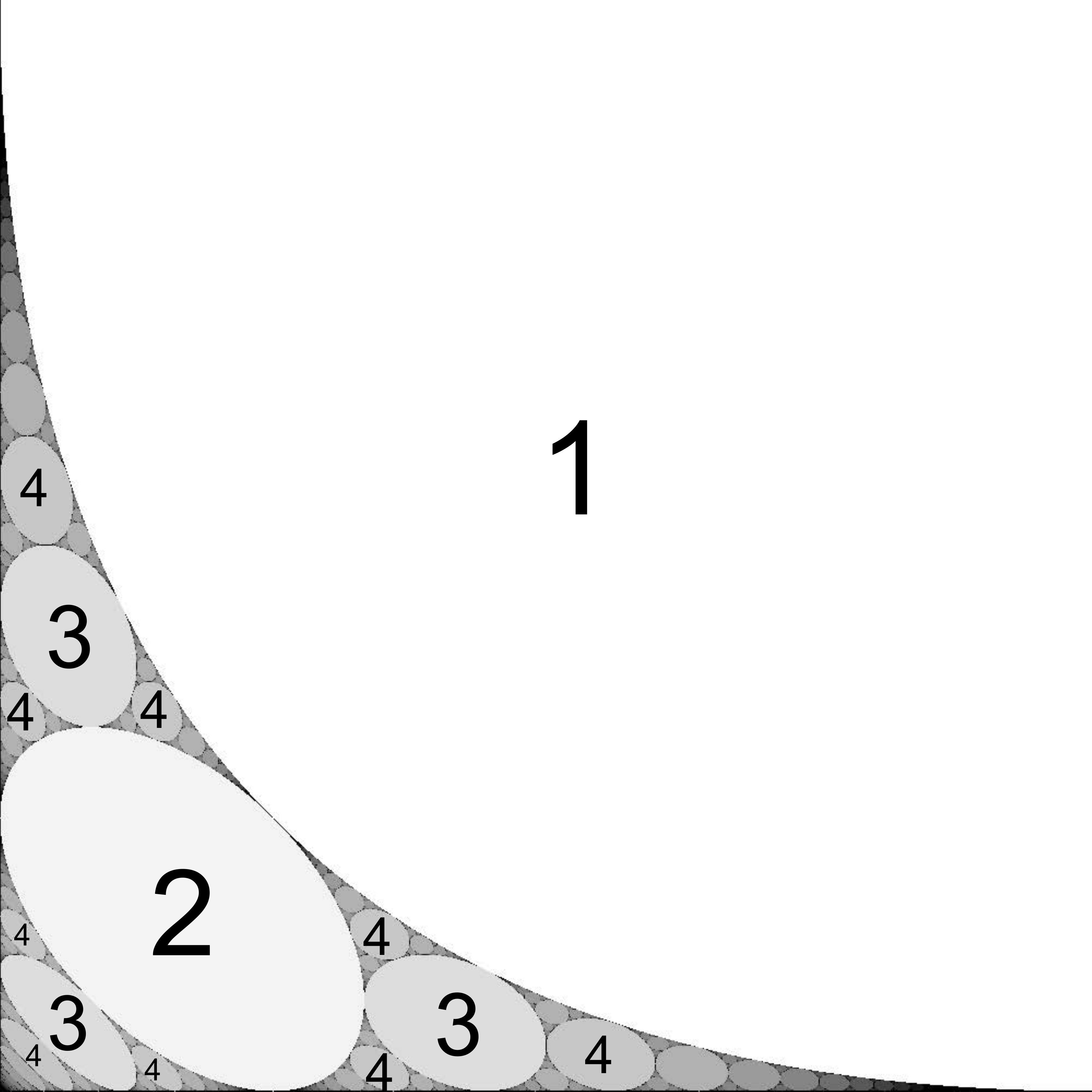}  
\caption{\small Left: Configuration space of tricycles.  Right: Plateaus of constant depth}
\label{fig:web}
\end{figure}
The fractal has a number of interesting properties, 
the most conspicuous being a deformed Stern-Brocot structure in the pattern of the points of tangency.
The ellipses in contact with the $x$ axis are tangent to it at the squares of rational numbers.  
$p^2/m^2$. 
If fractions $\frac{p}{m}$ and $\frac{q}{n}$ 
satisfy $pn-qm = \pm 1$
then the ellipses are mutually tangent and there is an ellipse inscribed between them at
\begin{equation}
\label{eq:x}
x=\frac{(p\!+\!q)^2}{(m\!+\!n)^2}\,.
\end{equation} 

A question arises: 
Is there a way to bring the ellipses to regular circles via some simple transformation? 
``Unsquaring'' the coordinates, suggested by the quadratic form of \eqref{eq:x}
is shown in Figure \ref{fig:bary}, left.
Although interesting artistically, it did not do the trick.
Changing the plot to barycentric coordinates also fails, as shown 
Figure \ref{fig:bary}, right.
%
\begin{figure}[h]
\centering
\includegraphics[scale=.09]{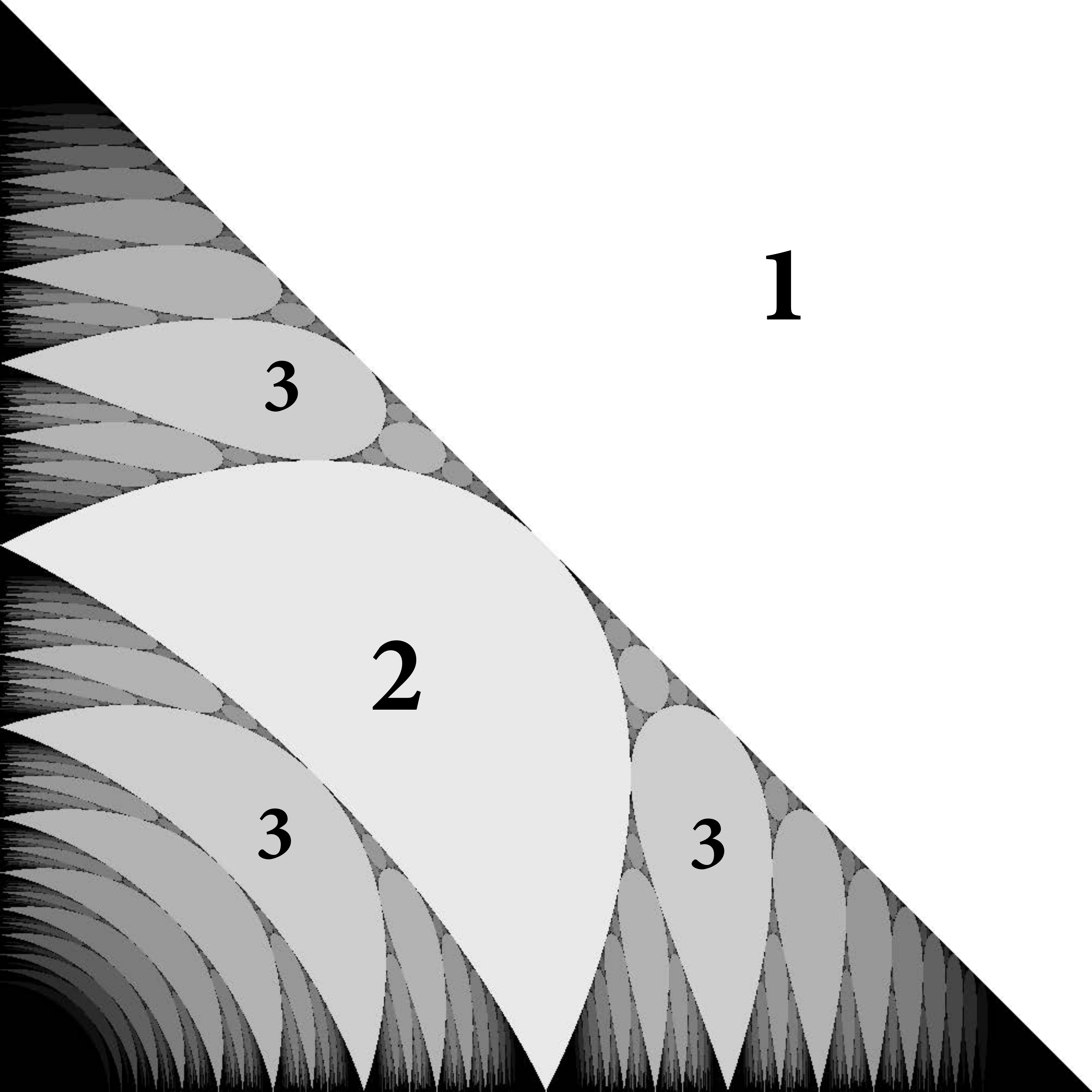}
\qquad\quad
\includegraphics[scale=.14]{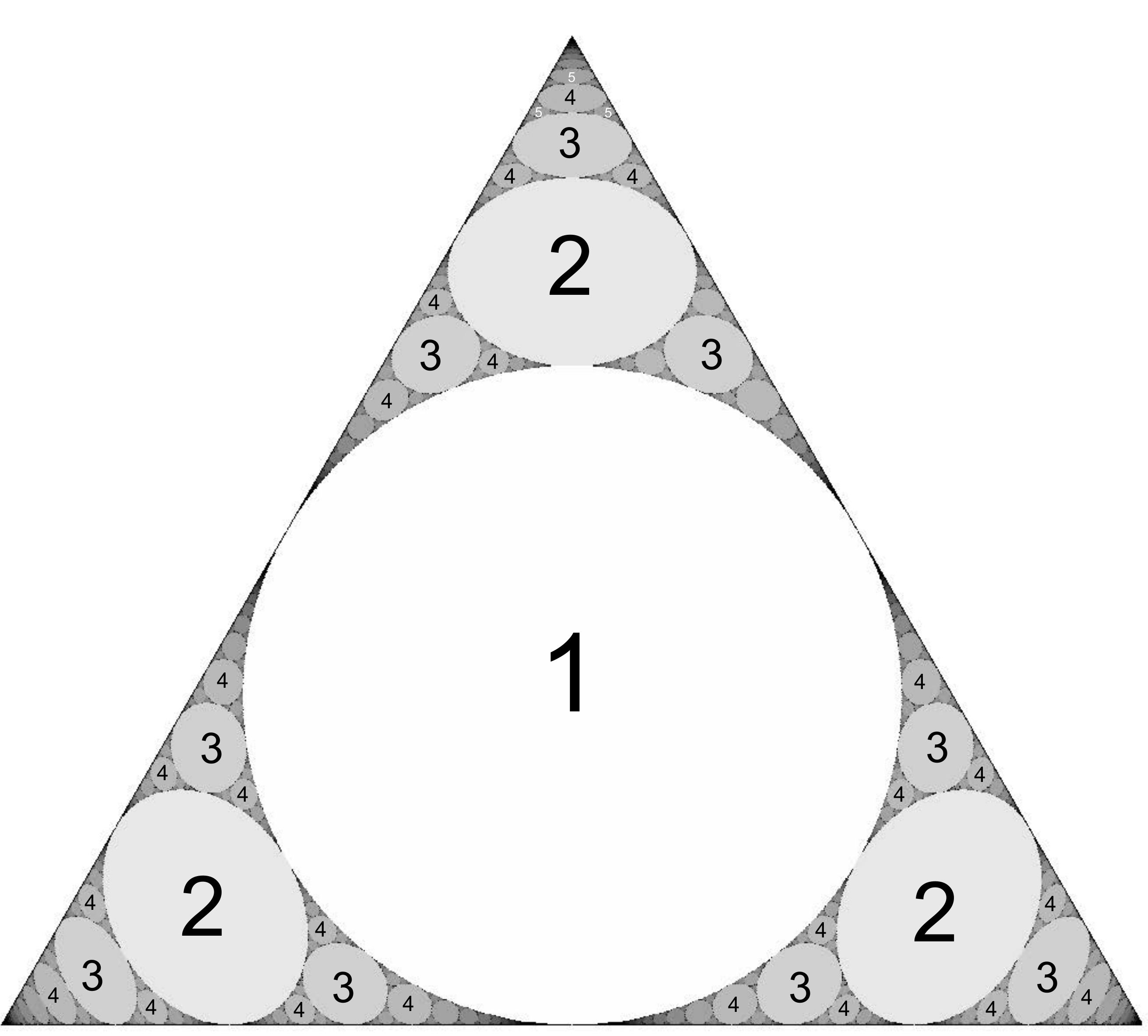} 
\caption{\small Left: quadratic deformation of the web. Right: Depth in barycentric representation.}
\label{fig:bary}
\end{figure}
In the present paper we employ the concept of tangency spinors, which, remarkably,
results with a parametrization of tricycles via Argand plane, 
in which the depth fractal assumes a regular shape, that of the Apollonian Belt.
The symmetries of the plane bring about a ``super-Dedekind'' tessellation,
an extended version of the Dedekind tessellation.
%
%
%
%
%
%
%
%
%
%
%
%
%
%




%
\newpage

\section{Pauli spinors and tricycles}

Recall that for a pair of two disks in contact in an Euclidean plane 
(the complex plane $\mathbb R^2 \cong \mathbb C$)
the {\bf tangency spinor} is defined as a 2-vector or equivalently a complex number
$$
\mathbf u = \begin{bmatrix}
                   x \\  y
                   \end{bmatrix}
\ = \  x+iy
\ = \ \tspin(A,B)
$$
such that 
$$
u^2 =  \frac{w}{r_1r_2}
$$
where $w$ is a complex number representing the vector joining the centers of the disks,
and $A=1/r_1$ and $B=1/r_2$ denote both the circles, and their curvatures.
Spinor is defined up to a sign. 
Recall that the arrows in the figures represent only the order of the disks (not the actual spinor).                
The remarkable properties of spinors are presented in 
\cite{jk-spinor}, and recapitulated in \cite{jk-lattice}
(For the first appearance, see \cite{jk-Clifford}) .

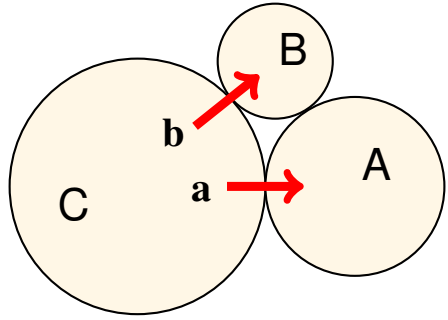
\begin{figure}[h]
\centering
\def\rc{1}
\def\ra{.7}
\def\rb{.45}   
\def\xb{1.077}
\def\yb{.98}
\begin{tikzpicture}[scale=1.7, rotate=0]

\draw [fill=gold!10, thick] (0, 0) circle (\rc);
\draw [fill=gold!10, thick] (\rc+\ra,0) circle (\ra);
\draw [fill=gold!10, thick] (\xb, \yb) circle (\rb);

\node [scale=1.6] at (-\rc/2, -\rc/7) {$\sf C$};
\node [scale=1.6] at (\rc+\ra+.17,   0 +.17) {$\sf A$};
\node [scale=1.6] at (\xb +.14,  \yb+.09) {$\sf B$};
%
\draw [red,->, line width=1.2mm] (\rc-.3, 0)--(\rc+.3, 0);
\node  [scale=1.4]  at (\rc-.5, 0-.04)  {$\mathbf a$};

\draw [red,->, line width=1.2mm] (.45, .47)--(.95, .87);
\node [scale=1.4]  at (.28, .42)  {$\mathbf b$};
\end{tikzpicture}
\caption{Tricycle and two spinors}
\label{fig:2spinors}
\end{figure}

Consider a tricycle with curvatures $A$, $B$, and $C$, 
 and spinors $\mathbf a=\tspin(c,a)$ and $\mathbf b=\tspin(C,B)$, as in Figure \ref{fig:2spinors}.
With the right choice of the signs of the spinors, we can write the following equations:

\begin{equation}
\label{eq:ABC}
\begin{array}{rl}
 (i)    &          \mathbf a \times \mathbf b = C\\
 (ii)  &        \|\mathbf a\|^2 = C+A \\
 (iii)  &        \|\mathbf b\|^2 = C+B \\
 \hspace{-1in}\hbox{and ~~~~~~~~~~~}\\
 (iv)  &       \| \mathbf a\pm \mathbf b\| = C+D_{\pm}\\
  (v)  &        \mathbf a\cdot \mathbf b = K
 \end{array}
 \end{equation}
from which we will need initially the first three.
The $D_{\pm}$ stands for two curvatures of the two disks complementing $(A,B,C)$ 
to the Descartes configuration, and $K$ stands for the curvature of the mid-circle 
that passes through the three points of tangency of $A$, $B$, and $C$.
(It is orthogonal to each of them.)

We may combine the two spinors into a single {\bf tangency Pauli spinor}, the vector 
$$
\xi = \begin{bmatrix}
                   a \\ b
                   \end{bmatrix}
\quad \in \ \mathbb C^2
$$
with $\mathbf a$ and $\mathbf b$ understood as complex numbers
(a vector which is much like the standard Pauli spinor for describing the electron's spin.)
If the tricycle is considered up to a scale and orientation, 
we may map $\xi$ into a single complex number:
\begin{equation}
\label{eq:project}
\xi = \begin{bmatrix}
                   a \\  b
                   \end{bmatrix}
\quad \xrightarrow{~~~} \quad
\begin{bmatrix}
                   1 \\  b/a
                   \end{bmatrix}
\quad \xrightarrow{~~~} \quad
b/a \ =\  z \ =\ x+iy
\end{equation}
Under this map, spinors $\mathbf a$ and $\mathbf b$ may be replaced by 
the following two (we keep the same names $\mathbf a$ and $\mathbf b$ not to multiply symbols used):
\begin{equation}
\label{eq:abclean}
\mathbf a = \begin{bmatrix}
                   1 \\   0
\end{bmatrix}\,,
\qquad
\mathbf b = \begin{bmatrix}
                   x \\  y
\end{bmatrix}
 \end{equation}
Using the associations \eqref{eq:ABC}, we get 
\begin{equation}
\label{eq:ABCxy}
\boxed{\qquad
\begin{array}{rl}
 (i)    &       C= y  \\
 (ii)  &        A = 1-y \\
 (iii)  &       B =  x^2 + y^2 - y \\
 \end{array}
 \qquad}
 \end{equation}
  
 \begin{proposition}
 The set of tricycles considered up to scaling and rotation may be parametrized
 by $z=x+iy$ via Eq \eqref{eq:ABCxy}, with the property  
\begin{equation}
\label{eq:dz}
\mathbf \delta(z) = \delta(A,B,C)  \,.
 \end{equation}
  \end{proposition}
 
Now,  we may execute a code for calculating the depth function as a function $\delta(x,y)$.
The result is startling and is presented for positive $x$ and $y$ in Figure \ref{fig:czarny} below.
The color coding is from black to blue to read as the value of depth grows.
The black regions correspond to zero depth.
The bottom line of the colorful belt coincides with the $x$-axis.
The left side of the figure coincides with the $y$-axis.

\begin{figure}[H]
\centering
\includegraphics[scale=.5]{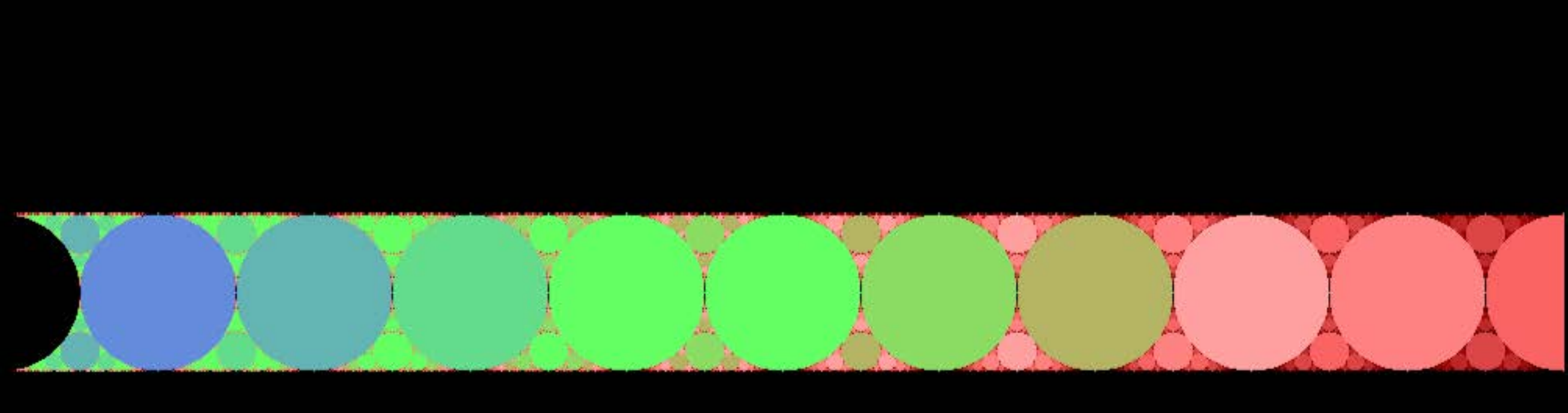}
\caption{\small The result of programming: the regions of different depth are disks.}
\label{fig:czarny}
\end{figure}

 We shall call the complex plane in this context the {\bf projective spinor space}.
 As presented above, it is the result of projectivization of the tangency Pauli spinors:
$$
\mathbb C\times \mathbb C  \  \xrightarrow{~~\oplus ~~}  \  
\mathbb C^2  \  \xrightarrow{~~\pi ~~}  \  
\mathbb C  {\rm P}^1  \  \xrightarrow{~~\cong~~}  \  
\dot{\mathbb C} \;\equiv\; \mathbb C\cup \{\infty\} 
$$

\begin{figure}[H]
\centering
\includegraphics[scale=.35]{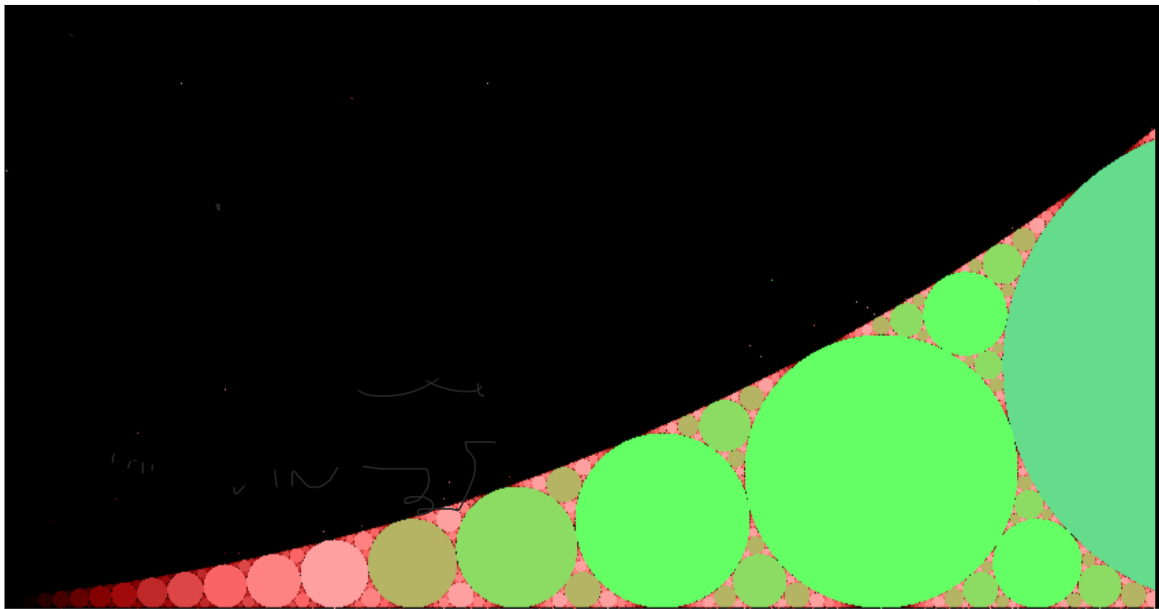}
\caption{\small Lower corner at the central disk of the depth plot. Lower left is point (0,0).}
\label{fig:corner}
\end{figure}

 Let us start with a few basic observations:
 
 \begin{itemize}
 \item
 Visually, the resulting fractal is similar to that of the Apollonian Belt (see Appendix B).
 In that sense we get a surprisingly regular pattern, unlike that of Figure \ref{fig:web} of \cite{jk-web}.
 \\[-19pt]
\item
Regions of arbitrarily  high values of the depth function exist  in contact with the regions 
of small values.  Figure \ref{fig:corner} shows a close-up of the corner region near the point (0,0).  

\item
The left-right symmetry is due to two mirror versions of regular tricycle (chiral versions).
\\[-19pt]
 \item
 The 0-depth regions correspond to the cases when one of the disks in a tricycle 
 is of negative curvature, and may be derived from \eqref{eq:ABCxy}:
  \\[-19pt]
 \begin{itemize}
 \item[$\bullet$]
 \ $C<0$ \ for \ $y<0$\\[-19pt]
 \item[$\bullet$]
 \ $A<0$ \ for \  $y>1$\\[-19pt]
 \item[$\bullet$]
 \ $B<0$ \ for \ $x^2+(y-1/2)^2<1/4$\\[-19pt]
 \end{itemize}

\item
The pattern of the depth values in the Apollonian belt follows the order of completing the 
initial three disks made by the 0-value regions, the disks bounded by circles
$y=0$, $y=1$ and $x^2+(y-1/2)^2 = 1/4$.
 \end{itemize}

 There is a number of conspicuous symmetries that we will explore next.

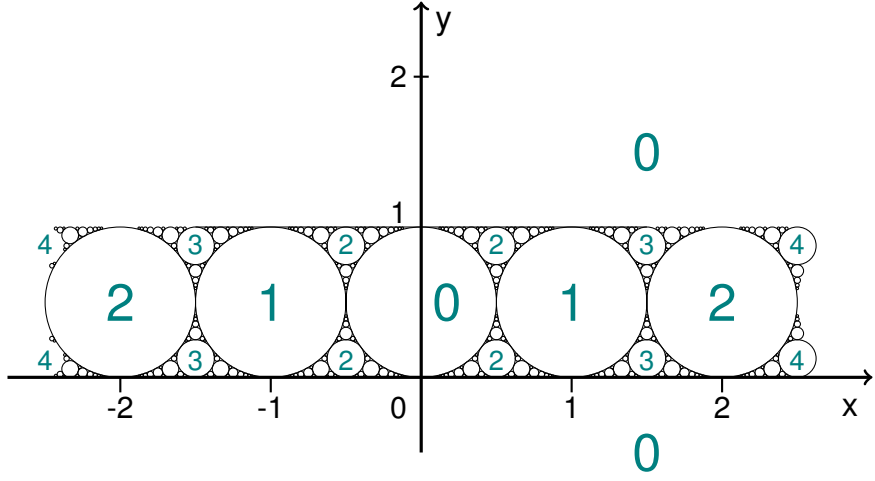
\begin{figure}[H]
\centering

\begin{tikzpicture}[scale=1, shift={(0,2cm)}]  
\draw [->, very thick ]  (-5.5,-1) -- (6,-1);
\node at (5.7,-1.4) [scale=1.2, color=black] {\sf x};
\draw (-3,1) -- (3,1);
\draw [->, very thick] (0,-2) -- (0,4);  
\node at (.3,3.7) [scale=1.2, color=black] {\sf y};

\draw [thick] (2,-1)--(2,-1.2); 
     \node at (2,-1.4) [scale=1.1, color=black] {\sf 1};
\draw [thick] (4,-1)--(4,-1.2); 
     \node at (4,-1.4) [scale=1.1, color=black] {\sf 2};
\draw [thick] (-2,-1)--(-2,-1.2); 
     \node at (-2,-1.4) [scale=1.1, color=black] {\sf -1};
\draw [thick] (-4,-1)--(-4,-1.2); 
     \node at (-4,-1.4) [scale=1.1, color=black] {\sf -2};

\draw [thick] (-.1, 3)--(.1,3); 
     \node at (-.3,3) [scale=1.1, color=black] {\sf 2};
     \node at (-.3,1.2) [scale=1.1, color=black] {\sf 1};
     \node at (-.3,-1.4) [scale=1.1, color=black] {\sf 0};

\draw (0,0) circle (1)   (2,0) circle (1)  (-2,0) circle (1)  (4,0) circle (1)  (-4,0) circle (1);

\foreach \b/\a/\c in {
3/4/4,  
5/12/12,  
7/24/24, 
9/40/40,  
11/60/60
%
}
\draw (\a/\c,\b/\c) circle (1/\c)     (\a/\c,-\b/\c) circle (1/\c)
%
(\a/\c+2,\b/\c) circle (1/\c)    (\a/\c+2,-\b/\c) circle (1/\c)  
(\a/\c+4,\b/\c) circle (1/\c)    (\a/\c+4,-\b/\c) circle (1/\c)  
(\a/\c-2,\b/\c) circle (1/\c)    (\a/\c-2,-\b/\c) circle (1/\c)  
(\a/\c-4,\b/\c) circle (1/\c)    (\a/\c-4,-\b/\c) circle (1/\c)  
         ;

\foreach \b/\a/\c in {
8/    6/   9, 	
15/   8/   16 , 	
24/  20/  25, 	
24/  10/  25, 	
21/  20/  28,
16/	30/	33,    
35/	12/	36,
48/	42/	49,
48/	28/	49,
48/	14/	49,
45/	28/	52,
40/	42/	57,
33/	56/	64,
63/	48/	64,
63/	16/	64
%
}
\draw (\a/\c, \b/\c) circle (1/\c)          (-\a/\c, \b/\c) circle (1/\c)
          (\a/\c,-\b/\c) circle (1/\c)         (-\a/\c,-\b/\c) circle (1/\c)
          (\a/\c+2, \b/\c) circle (1/\c)          (-\a/\c+2, \b/\c) circle (1/\c)
          (\a/\c+2,-\b/\c) circle (1/\c)         (-\a/\c+2,-\b/\c) circle (1/\c)
          (\a/\c+4, \b/\c) circle (1/\c)          (-\a/\c+4, \b/\c) circle (1/\c)
          (\a/\c+4,-\b/\c) circle (1/\c)         (-\a/\c+4,-\b/\c) circle (1/\c)
          (\a/\c-2, \b/\c) circle (1/\c)          (-\a/\c-2, \b/\c) circle (1/\c)
          (\a/\c-2,-\b/\c) circle (1/\c)         (-\a/\c-2,-\b/\c) circle (1/\c)
          (\a/\c-4, \b/\c) circle (1/\c)          (-\a/\c-4, \b/\c) circle (1/\c)
          (\a/\c-4,-\b/\c) circle (1/\c)         (-\a/\c-4,-\b/\c) circle (1/\c)
;
\node at (.34,0) [scale=1.9, color=teal] {\sf 0};
\node at (-2,0) [scale=1.9, color=teal] {\sf 1};
\node at (2,0) [scale=1.9, color=teal] {\sf 1};
\node at (-4,0) [scale=1.9, color=teal] {\sf 2};
\node at (4,0) [scale=1.9, color=teal] {\sf 2};

\node at (3,2) [scale=1.9, color=teal] {\sf 0};
\node at (3,-2) [scale=1.9, color=teal] {\sf 0};

\node at (1,0.77) [scale=1, color=teal] {\sf 2};
\node at (-1,0.77) [scale=1, color=teal] {\sf 2};
\node at (1,-0.77) [scale=1, color=teal] {\sf 2};
\node at (-1,-0.77) [scale=1, color=teal] {\sf 2};

\node at (3,0.77) [scale=1, color=teal] {\sf 3};
\node at (-3,0.77) [scale=1, color=teal] {\sf 3};
\node at (3,-0.77) [scale=1, color=teal] {\sf 3};
\node at (-3,-0.77) [scale=1, color=teal] {\sf 3};

\node at (5,0.77) [scale=1, color=teal] {\sf 4};
\node at (-5,0.77) [scale=1, color=teal] {\sf 4};
\node at (5,-0.77) [scale=1, color=teal] {\sf 4};
\node at (-5,-0.77) [scale=1, color=teal] {\sf 4};

\end{tikzpicture}
\caption{the depth structure of the experimental fractal}
\label{fig:Apollo}
\end{figure}

\section{Finite symmetries}

The symmetries that transfer the Apollonian belt to itself and permute 
the  regions of depth 0  are easy to spot, see Figure \ref{fig:symmetries}.
They are:\\[-7pt]

(1) \  $F:$  Reflection through the horizontal line $y=1/2$.

(2) \  $S:$ Inversion through the unit circle centered at origin, \ $x^2+y^2=1$.

(3) \  $R:$ The Inversion through the unit  circle centered at (0,1), \ $x^2+(y\!-\!1)^2=1$.

~

Here is a simple observation:

\begin{proposition}
\label{thm:FSR}
The above three transformations permute the values of curvatures among the three circles of the tricycles.
$$
F: C\leftrightarrows A
\qquad
S: A\leftrightarrows B
\qquad
R: B\leftrightarrows C
$$
and preserve the patter of the depth function in Fig. \ref{fig:czarny}.
\end{proposition}

\begin{proof}
The coordinate description $(x,y)\to(x',y')$above transformations are:
\begin{equation}
\label{eq:FSR}
 F: \  \left\{\begin{array}{l}
         x' = x\\
         y'=1-y
        \end{array}\right.
\qquad        
S: \  \left\{\begin{array}{l}
         x' = \dfrac{x}{x^2+y^2}\\
         y'=\dfrac{y}{x^2+y^2}
        \end{array}\right.   
\qquad        
R: \  \left\{\begin{array}{l}
         x' = \dfrac{x}{x^2+(y-1)^2}\\
         y'=\dfrac{x^2+y^2-y}{x^2+(y-1)^2}
        \end{array}\right.                     
\end{equation}
As to the reflection $F$, the claim is obvious.  For inversion $S$, calculate:
$$
\begin{array}{clll}
C\ \mapsto \  C' &= y' = \dfrac{x^2+y^2-y}{x^2+(y-1)^2} = \dfrac{C}{x^2+(y-1)^2} 
&  \sim C
\\[12pt]
A\ \mapsto \ A'  &= 1-y' =  1- \dfrac{x^2+y^2-y}{x^2+(y-1)^2}  = \dfrac{B}{x^2+(y-1)^2} 
&\sim B
\\[12pt]
B \ \mapsto \  B' &= x'^2 + y'^2 - y' = ... = \dfrac{C}{x^2+(y-1)^2}
&\sim C
\end{array}
$$
Thus all prove to be re-scaled by the same factor.
For the third transformation, $R$, similar calculations show that 
$$
C' =  \dfrac{C}{x^2+(y-1)^2}, 
\qquad  
A' =  \dfrac{B}{x^2+(y-1)^2},
\qquad
B' = \dfrac{C}{x^2+(y-1)^2}  
$$
Thus the transformations $F$, $S$, $R$ preserve the mutual ratios of the the curvatures, 
and consequently preserve the depth, hence the the conclusion of the invariance of the pattern.
\end{proof}

Yet another apparent feature of the image in Figure \ref{fig:Apollo} is the left-right mirror symmetry:

~

\begin{figure}
\centering

\begin{tikzpicture}[scale=.7, shift={(0,2cm)}]  
\draw [->, very thick ]  (-5.5,-1) -- (6,-1);
\node at (5.7,-1.4) [scale=1.2, color=black] {\sf x};
\draw (-3,1) -- (3,1);
\draw [->, very thick] (0,-3.3) -- (0,4.5);  
\node at (-.4,4.7) [scale=1.2, color=black] {\sf y};

\draw [thick] (2,-1)--(2,-1.2); 
     \node at (2.3,-1.44) [scale=1.1, color=black] {\sf 1};
\draw [thick] (4,-1)--(4,-1.2); 
     \node at (4,-1.44) [scale=1.1, color=black] {\sf 2};
\draw [thick] (-2,-1)--(-2,-1.2); 
     \node at (-2.3,-1.44) [scale=1.1, color=black] {\sf -1};
\draw [thick] (-4,-1)--(-4,-1.2); 
     \node at (-4,-1.44) [scale=1.1, color=black] {\sf -2};
\draw [thick] (-.1, 3)--(.1,3); 
     \node at (-.3,2.5) [scale=1.1, color=black] {\sf 2};
     \node at (-.3,1.4) [scale=1.1, color=black] {\sf 1};
     \node at (-.3,-1.4) [scale=1.1, color=black] {\sf 0};
\draw (0,0) circle (1)   (2,0) circle (1)  (-2,0) circle (1)  (4,0) circle (1)  (-4,0) circle (1);
\foreach \b/\a/\c in {
3/4/4,  
5/12/12,  
7/24/24, 
9/40/40,  
11/60/60
%
}
\draw (\a/\c,\b/\c) circle (1/\c)     (\a/\c,-\b/\c) circle (1/\c)
%
(\a/\c+2,\b/\c) circle (1/\c)    (\a/\c+2,-\b/\c) circle (1/\c)  
(\a/\c+4,\b/\c) circle (1/\c)    (\a/\c+4,-\b/\c) circle (1/\c)  
(\a/\c-2,\b/\c) circle (1/\c)    (\a/\c-2,-\b/\c) circle (1/\c)  
(\a/\c-4,\b/\c) circle (1/\c)    (\a/\c-4,-\b/\c) circle (1/\c)  
         ;

\foreach \b/\a/\c in {
8/    6/   9, 	
15/   8/   16 , 	
24/  20/  25, 	
24/  10/  25, 	
21/  20/  28,
16/	30/	33,    
35/	12/	36,
48/	42/	49,
48/	28/	49,
48/	14/	49,
45/	28/	52,
40/	42/	57,
33/	56/	64,
63/	48/	64,
63/	16/	64
%
}
\draw (\a/\c, \b/\c) circle (1/\c)          (-\a/\c, \b/\c) circle (1/\c)
          (\a/\c,-\b/\c) circle (1/\c)         (-\a/\c,-\b/\c) circle (1/\c)
          (\a/\c+2, \b/\c) circle (1/\c)          (-\a/\c+2, \b/\c) circle (1/\c)
          (\a/\c+2,-\b/\c) circle (1/\c)         (-\a/\c+2,-\b/\c) circle (1/\c)
          (\a/\c+4, \b/\c) circle (1/\c)          (-\a/\c+4, \b/\c) circle (1/\c)
          (\a/\c+4,-\b/\c) circle (1/\c)         (-\a/\c+4,-\b/\c) circle (1/\c)
          (\a/\c-2, \b/\c) circle (1/\c)          (-\a/\c-2, \b/\c) circle (1/\c)
          (\a/\c-2,-\b/\c) circle (1/\c)         (-\a/\c-2,-\b/\c) circle (1/\c)
          (\a/\c-4, \b/\c) circle (1/\c)          (-\a/\c-4, \b/\c) circle (1/\c)
          (\a/\c-4,-\b/\c) circle (1/\c)         (-\a/\c-4,-\b/\c) circle (1/\c)
;
\node at (-.34,0) [scale=1.9, color=teal] {\sf 0};
\node at (-2,0) [scale=1.9, color=teal] {\sf 1};
\node at (2,0) [scale=1.9, color=teal] {\sf 1};
\node at (-4,0) [scale=1.9, color=teal] {\sf 2};
\node at (4,0) [scale=1.9, color=teal] {\sf 2};

\node at (-3,2.4) [scale=1.9, color=teal] {\sf 0};
\node at (-3,-2.4) [scale=1.9, color=teal] {\sf 0};

\node at (1,0.77) [scale=1, color=teal] {\sf 2};
\node at (-1,0.77) [scale=1, color=teal] {\sf 2};
\node at (1,-0.77) [scale=1, color=teal] {\sf 2};
\node at (-1,-0.77) [scale=1, color=teal] {\sf 2};

\node at (3,0.77) [scale=1, color=teal] {\sf 3};
\node at (-3,0.77) [scale=1, color=teal] {\sf 3};
\node at (3,-0.77) [scale=1, color=teal] {\sf 3};
\node at (-3,-0.77) [scale=1, color=teal] {\sf 3};

\draw [very thick, color=red] (0,-1) circle (2);
\draw [very thick, color=red] (0,1) circle (2);
\draw [very thick, color=red] (-5.5,0) -- (6.5,0);

\draw [<->, line width=.8mm, color=blue] (3/4,2) --(3/4+1, 2+1);
\node at (2.3,3.2) [scale=1.2, color=blue] {\sf R};

\draw [<->, line width=.8mm, color=blue] (3/4,0-2) --(3/4+1, 0-2-1);
\node at (2.3,-3.2) [scale=1.2, color=blue] {\sf S};

\draw [<->, line width=.8mm, color=blue] (6,  -.6) --(6, .6);
\node at (6.2,1) [scale=1.2, color=blue] {\sf F};

\draw [<->, line width=.8mm, color=blue] (-.6,  3.5) --(.6, 3.5);
\node at (.7,4) [scale=1.2, color=blue] {\sf M};

\end{tikzpicture}
\caption{Symmetries of the Apollonian depth fractal}
\label{fig:symmetries}
\end{figure}
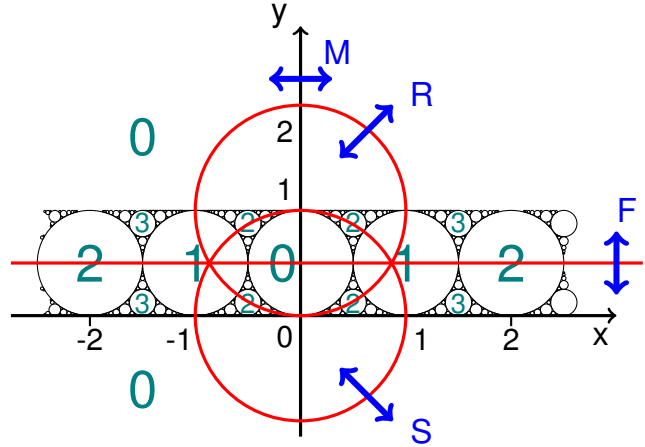

(4) \ $H$:  Reflection through the vertical axis $x=0$:
\begin{equation}
\label{eq:H}
 H: \  \left\{\begin{array}{l}
         x' = -x\\
         y'=y
        \end{array}\right.
\end{equation}
It does not affect the Equations \eqref{eq:ABCxy}, but  geometrically it represents the mirror reflection 
of the tricycles.

Denote the two finite groups generated by these transformations 
(with and without symmetry $H$): 
\begin{equation}
\label{eq:groupsfinite}
\Theta_{\sf o} = \gen\{S,F\}
\qquad\qquad
\Theta = \gen\{S, F,H \}
\end{equation}
Note that both groups contain also inversion $R=SFS=FSF\in \Theta_{\sf o}$.  
The defining identities are:
$$
H^2= S^2 = F^2=R^2 = (SF)^3=    (FR)^3=(RS)^3=\id \,.
$$
Moreover $H$ commutes with all other elements.
Group $\Theta$ has 12 elements and splits $\dot{\mathbb C}$ into 12 regions, each
of which nay serve as the fundamental domain for group $\Theta$,
i.e. a region that contains exactly one element of every orbit of the action defined by the group.
Figure \ref{fig:12} show the situation with a particular choice of the fundamental domain,
marked as $Q$. The fundamental region for group $\Theta_0$ may be chosen as $Q\cup HQ$.
\\

The essence of Proposition \ref{thm:FSR}  may be now reformulated;

\begin{proposition}
\label{thm:theta}
The groups of symmetry \eqref{eq:groupsfinite} preserve the depth function:
$$
\forall g\in \Theta \qquad  \delta(gz) \ = \ \delta(z)
$$
This results in characterization of the quotients of the action:
$$
\begin{array}{cl}
\mathbb Z/\Theta_{\sf o} &= \
            \left\{ \, {\hbox{\rm all tricycles up to similarity}
            \atop 
            \hbox{\rm  excluding reflections \hfill}}\,\right\}
            \ \cong \  Q\cup HQ
\\[12pt]
\mathbb Z/\Theta \  &= \
            \left\{ \, {\hbox{\rm all tricycles up to similarity}
            \atop 
            \hbox{\rm including reflections \hfill}}\,\right\}
            \ \cong \  Q
\end{array}
$$
\end{proposition}

The latter, which disregards the chirality, 
coincides with the orbits of the algebraic definition of the process 
for triples as numbers.
The fundamental region represents all tricycles (up to permutations etc.).
Note that every depth (color) is appears in these fundamental domains.

The M{\"o}bius representation of the group elements will be given 
in the next section. 

\begin{figure}[H]
\centering
\includegraphics[scale=.77]{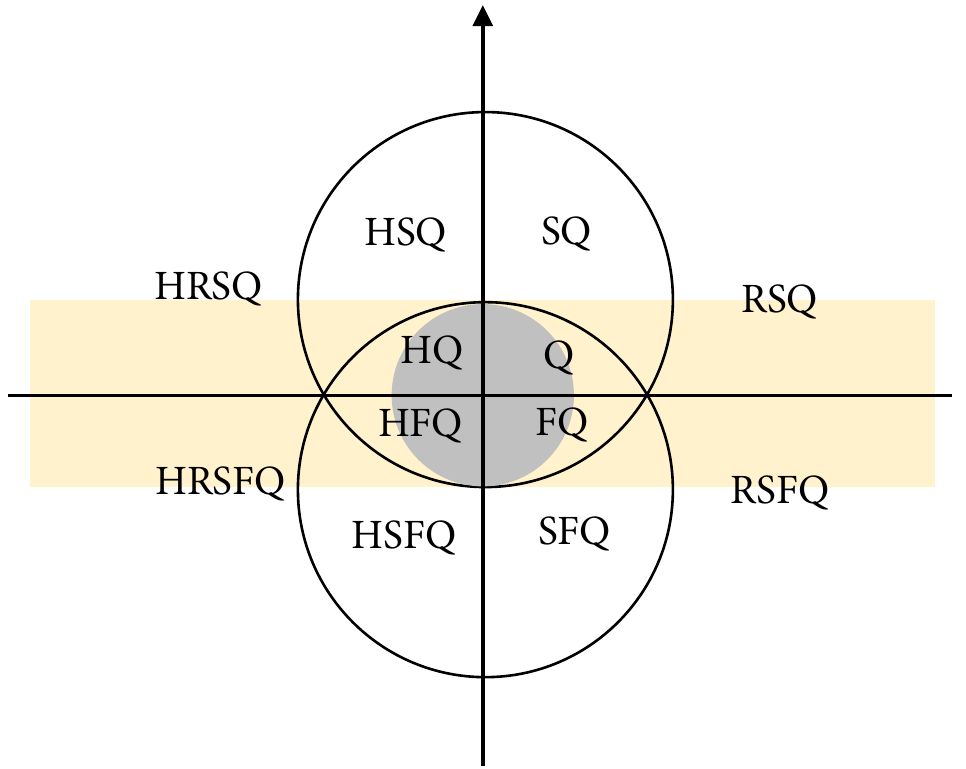}
\caption{\small Twelve regions of the tessellation $\mathbb C/\Theta$.  
Region $Q$ is chosen as the fundamental domain.
The $x$-axis passes through the center of the lower circle. }
\label{fig:12}
\end{figure}

\def\AAA{\mathlarger{\mathcal A}}

\newpage
\section{Apollonian packings and super-Dedekind tessellation}

The projective space of the tangency Pauli spinors considered above may be also used to
classify the Apollonian disk packing.
This will lead us to yet another, infinite, partition of the plane, 
called below the {\bf super-Dedekind tessellation}.
Recall that every tricycle $(A,B,C)$ defines uniquely an Apollonian disk packing $\AAA(A,B,C)$.
Define an equivalence relation which equates tricycles 
if they generate the same Apollonian disk packing (as usual, up to similarity): 
$$
 (A,B,C) \; \sim\; (A', B', C')  \qquad \hbox{if} \qquad \AAA(A,B,C)\; \cong\; \AAA(A'\!, B'\!, C') 
$$
We may execute this equivalence by adding yet another element of symmetry, namely
$$
         T: z \mapsto z +1   
$$

\def\rc{1}
\def\ra{.7}
\def\rb{.45}   
\def\xb{1.077}
\def\yb{.98}
\begin{figure}[h]
\centering
\begin{tikzpicture}[scale=2, rotate=0]
\clip (-1.45,-1.1) rectangle (2.8,1.7);

\draw[fill=gold!42] (-2,-2) rectangle (3,2);
\draw [thick, fill=black] (.7,-.247) circle (1.7464);
\draw [fill=gold!21, thick] (0, 0) circle (\rc);
\draw [fill=gold!21, thick] (\rc+\ra,0) circle (\ra);
\draw [fill=gold!21, thick] (\xb, \yb) circle (\rb);
\draw [fill=gold!42, thick] (\xb-.057, .43) circle (.1);
\draw [thick] (.7,-.247) circle (1.7464);
\node [scale=1.7] at (-\rc/7, -\rc/2) {$C$};
\node [scale=1.7] at (\rc+\ra+.17,   0 +.17) {$A$};
\node [scale=1.7] at (\xb +.14,  \yb+.09) {$B$};
%
\draw [red,->, line width=1.2mm] (\rc-.3, 0)--(\rc+.3, 0);
\node  [scale=1.4]  at (\rc-.4, -.15)  {$\mathbf a$};

\draw [red,->, line width=1.2mm] (.47, .47)--(.97, .91);
\node [scale=1.4]  at (.28, .6)  {$\mathbf b$};

\draw [blue,->, line width=.8mm] (.55, .23)--(.9, .38);
\node [scale=1]  at (.2, .18)  {$\mathbf a+\mathbf b$};

\draw [blue,->, line width=.8mm] (-.7, .2)--(-1.15, .32);
\node [scale=1]  at (-.7, -.02)  {$\mathbf a - \mathbf b$};
\end{tikzpicture}
\caption{Beyond a tricycle via spinors}
\label{fig:a+b}
\end{figure}
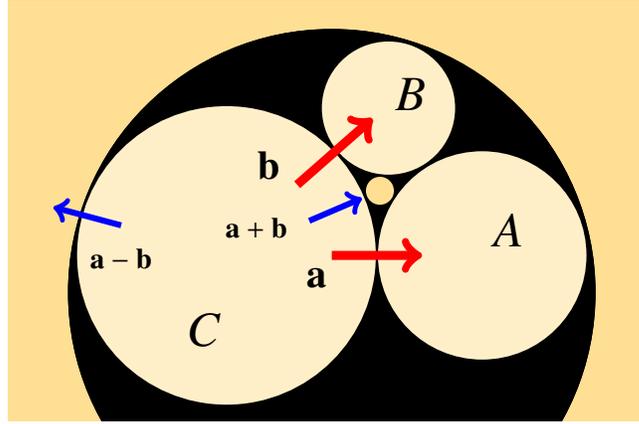

%

Here is why.
Recall that the sum or difference of the spinors $\mathbf a$ and $\mathbf b$ 
is the spinor from $C$ to the disk that is inscribed between the two (see Figure \ref{fig:a+b}).
The two cases correspond to the two ways of completing a tricycle to the Descartes configuration

\begin{equation}
\label{eq:abclean2}
\mathbf a = \begin{bmatrix}
                   1 \\   0
\end{bmatrix},   \ 
\mathbf b = \begin{bmatrix}
                   x \\  y
\end{bmatrix}
 \qquad
 \Rightarrow
 \qquad
\mathbf a = \begin{bmatrix}
                   1 \\   0
\end{bmatrix},  \ 
 \mathbf b\!\pm\!\mathbf a
 =
 \begin{bmatrix}
                   x\pm 1\\  y
\end{bmatrix}
 \end{equation}
In case of the spinors adjusted to the form \eqref{eq:ABCxy},
this corresponds to  in the space of spinors to shift to the left or right by one unit.
Hence we define transformation:
$$
T:z \ \mapsto \ z+1
$$
With the help of the other transformations that permute the disks, we recover 
all disks in the corresponding Apollonian packing.
This leads to infinite extension of the finite groups \eqref{eq:groupsfinite}.
As before, we have two versions: with or without the horizontal reflection $H$.
\begin{definition}
The {\bf Apollonian symmetry groups} are groups generated by the following elements
\begin{equation}
\label{eq:xi}
\Xi_{\sf o} = \gen\{T, S, F\}
\qquad\qquad
\Xi = \gen\{T, S, F,H\}
\end{equation}
\end{definition}

The orbits through any point (tricycle) consists of
all other tricycles (up to scale) in the Apollonian gasket it generates.
(Including and excluding the mirror versions due to the presence or its lack of $H$ in the group.)

~

The tessellations due to the action of the group $\Xi$ are now fragmenting the complex space into 
an infinite set of triangles and is shown in Figure \ref{fig:super}.
We shall call this pattern the {\bf super-Dedekind tessellation} for the reasons explained below.

Analogously to Proposition \ref{thm:theta}, we have now
\begin{proposition}
\label{thm:xi}
The action of the groups \eqref{eq:xi} 
preserve the generated Apollonian disk packings:
$$
\forall g\in \Xi \qquad  \mathcal A(gz) \ = \ \mathcal A(z)
$$
The corresponding quotients of the action are: 
$$
\begin{array}{cl}
\mathbb Z/\Xi_{\sf o} & = \ 
            \left\{ \, {\hbox{\rm all Apollonian packings up to similarity}
            \atop 
            \hbox{\rm  excluding reflections \hfill}}\,\right\}
            \ \cong \  P\cup HP
\\[12pt]
\mathbb Z/\Xi\  &= \ 
            \left\{ \, {\hbox{\rm all Apollonian packings up to similarity}
            \atop 
            \hbox{\rm  including reflections \hfill}}\,\right\}
            \ \cong \  P
\end{array}
$$
where the fundamental domain for the action of $\Xi$ 
may be chosen for instance 
$$
P \ = \ \{ \ (x,y)\in \mathbb R^2 \; |\; 0\leq x \leq 1/2, \, y \leq 0,\;  {\rm and} \ x^2+y^2\geq 1\} \,.
$$
\end{proposition}

\noindent
{\bf Description} 
\\

1. The super-Dedekind tessellation 
splits the plane $\mathbb C$
into triangles shown in Figure \ref{fig:super}.
It consists of triangles $(60^\circ, 90^\circ,0^\circ)$,
each of which may serve as the fundamental domain of the action of group $\Xi$.
The vertices have three types of valencies: 4, 6, $\infty$
The ones of valency $\infty$ are located along the circles of the Apollonian Belt
drawn in red in Figure \ref{fig:super}.
The circles are those of the boundaries of the regions of constant depth $\delta$
of the depth fractal.
\\

2. The pints of the triangles  are in 1-to-1 correspondence with all Apollonian disk packings (up to similarity). 
The orbit of any of the points in it travels through 
all tricycles contained in this packing (again, up to similarity).
\\

3. The super-Dedekind tessellation may be viewed as superimposing copies of the standard 
Dedekind tessellation.  Indeed, the lower part in the plane, $y<0$, as well as the upper part, $y>1$,
both are congruent to the Dedekind tessellations of Poincar\'e half-plane
(see \cite{LeB} and \cite{jk-dedekind}.
Note that the central disk of the underlying Apollonian Belt is tessellated by 
by the Dedekind tessellation of the unit disk.
In fact, every red circle of $\mathcal A_B$
contains an appropriate inversion of the Dedekind tessellation.
\\

4. If the mirror symmetry $H$ is excluded from the group,
the tessellation is coarser and consists of triangles $(60^\circ, 60^\circ,0^\circ)$,
made by gluing the former triangles.
The upper and lower part of the tessellation will now correspond to the well-known 
tessellation of the Poincar\'e upper half-plane resulting by the action of the group $\SL(2,\mathbb Z)$,
which in \cite{jk-dedekind} is called modular tessellation.
\\

5. Finally, note that the upper (and the lower) part  of the tessellation coincides with the results of \cite{jk-lattice}
where the integral packings were analyzed with the help of the tangency Pauli spinors.
\\

6. Algebraically, the super-Dedekind fractal is an extended version of the modular and Dedekind tessellations.
determined by the group $\SL(2,\mathbb C$).
Similarly, the super-modular fractal is obtained.
Both are obtained by adding an extra element to the generators, see Figure \ref{fig:groups}.

~\\
{\bf Remark:} The group $\Xi$ has yet another meaning: it is a group of symmetries
of the Apollonian Belt.
The basic element are:

\begin{enumerate}
\item
The inversions in 3-circles: circles that go through the points of contact of any three disks 
forming a tricycle in the disk packing.
\item
The inversions in 2-circles: circles that are determined by any two circles in contact,
namely as the circles that exchange them through the inversion.
They also pass through the contact points of the (infinite ) chain of circles
inscribed between the two circles.
Such a chain may be called a (generalized) {\bf Pappus chain}.  
\end{enumerate}

The proposed name ``Apollonian symmetry group'' is thus justified.
The group of concrete symmetries of other Apollonian packing is isomorphic to $\Xi$.
This matter will be discussed elsewhere.

~

 \def\FF{\phantom{T}}

\begin{figure}
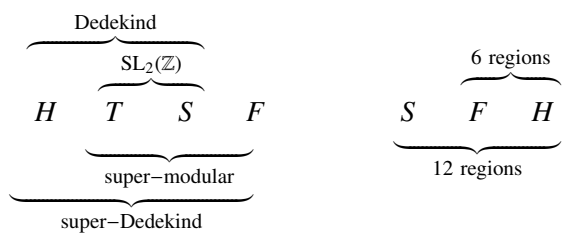

$$
\overbrace{\ H \ \quad  \overbrace{\ T \phantom{\|^{|}} \quad \ S\  }^{{\rm SL}_2(\mathbb Z)}}^{\rm Dedekind}  \quad \  F\  
\qquad\qquad 
\underbrace{ \ S \ \quad \overbrace{\ F \phantom{\Big|} \quad  \ H \ }^{\rm 6 \ regions}}_{\rm 12\ regions}  
$$
\vspace{-45pt}
$$
 \underbrace{ \ \; \FF\  \quad   \underbrace{\ \FF\ \quad \ \FF\ \quad\ \FF\ \; }_{\rm super-modular}}_{\rm super-Dedekind}\quad \
 \phantom{\qquad\qquad
 \underbrace{ \ S \ \quad \overbrace{\ F \ \quad  \ H \ }^{\rm 6666  }}_2 }  
$$
\caption{The groups discussed and their generators}
\label{fig:groups}
\end{figure}

\begin{figure}[H]
\centering
\begin{tikzpicture}[scale=3.7]
\clip (-2,-3) rectangle (2.01,1.5);

\draw (-2.5,-.5)--(2.5, -.5);
\foreach \n in {0, 1, 2, 3, 4}
\draw (\n + 0.5, -2.5)--(\n + 0.5, 1.5) 
          (-\n - 0.5, -2.5)--(-\n - 0.5, 1.5); 
\foreach \n in {0, 1, 2, 3, 4, 5}
\draw[color=blue] (\n , -2.5)--(\n , 1.5) 
          (-\n , -2.5)--(-\n , 1.5);           
          
\foreach \n in {-2, -1, 0, 1, 2}
\draw (\n , 0) circle (1/1)
         (\n , -1) circle (1/1);
\foreach \n in {-3,-2, -1, 0, 1,2, 3}
\foreach \a/\b   in {  
1/3,  1/5,  1/7,  1/9,  1/11,  1/13,  1/15, 1/17,  1/19,  1/21,   1/23,  
1/25,  1/27,  1/29,  1/31,  1/33, 1/35, 
  1/37,  1/39,   1/41,  1/43,
1/45,  1/47,  1/49 
 4/15,  8/21, 10/33,   6/35, 14/39, 19/45 
 }
\draw (\n + \a/\b, 0) circle (1/\b)
           (\n - \a/\b, 0) circle (1/\b)
          (\n + \a/\b, -1) circle (1/\b)
           (\n - \a/\b, -1) circle (1/\b)
           ;
\foreach \n in {-3,-2, -1, 0, 1,2, 3}
\foreach \a/\b   in {  
3/8,  7/16,  5/24,  11/24,
15/32,  11/40,  19/40,  7/48,  23/48   
}
\draw [color=black] (\n + \a/\b, 0) circle (1/\b)
           (\n - \a/\b, 0) circle (1/\b)
          (\n + \a/\b, -1) circle (1/\b)
           (\n - \a/\b, -1) circle (1/\b)
           ;

\foreach \n in {-3,-2, -1, 0, 1,2, 3}
\foreach \a/\b   in {  
1/2,  1/4, 1/6, 1/8, 1/10, 1/12, 1/14, 1/16, 1/18 , 1/20, 1/22, 1/24, 
1/26, 1/28,  1/30,  1/32, 1/34, 1/36 
,1/38, 1/40,  1/42, 1/44, 1/46, 1/48
}
\draw[color=blue] (\n + \a/\b, 0) circle (1/\b)
                         (\n - \a/\b, 0) circle (1/\b)      
                         (\n + \a/\b, -1) circle (1/\b)
                         (\n - \a/\b, -1) circle (1/\b)      
                              ;

\foreach \n in {-3,-2, -1, 0, 1,2, 3}
\foreach \a/\b   in {  
5/12, 9/20, 7/24, 13/28, 11/30, 17/36, 9/40, 13/42, 21/44, 17/48
}
\draw[color=blue] (\n + \a/\b, 0) circle (1/\b)
                           (\n - \a/\b, 0) circle (1/\b)           
                           (\n + \a/\b, -1) circle (1/\b)
                           (\n - \a/\b, -1) circle (1/\b)         
;

\foreach \a/\b/\c   in {  
2	/	1	/	1	,
2	/	2	/	2	,
2	/	3	/	3	,
2	/	4	/	4	,
4	/	2	/	4	,
2	/	5	/	5	,
2	/	6	/	6	,
6	/	2	/	6	,
2	/	7	/	7	,
2	/	8	/	8	,
8	/	2	/	8	,
6	/	7	/	9	,
2	/	9	/	9	,
2	/	10	/	10	,
10	/	2	/	10	,
2	/	11	/	11	,
10	/	5	/	11	,
2	/	12	/	12	,
12	/	2	/	12	,
2	/	13	/	13	,
10	/	10	/	14	,
2	/	14	/	14	,
14	/	2	/	14	,
2	/	15	/	15	,
8	/	14	/	16	,
2	/	16	/	16	,
16	/	2	/	16	,
14	/	8	/	16	,
2	/	17	/	17	,
2	/	18	/	18	,
18	/	2	/	18	,
14	/	13	/	19	,
2	/	19	/	19	,
2	/	20	/	20	,
20	/	2	/	20	,
2	/	21	/	21	,
18	/	11	/	21	,
2	/	22	/	22	,
22	/	2	/	22	,
2	/	23	/	23	,
22	/	7	/	23	,
18	/	16	/	24	,
16	/	18	/	24	,
2	/	24	/	24	,
24	/	2	/	24	,
10	/	23	/	25	,
2	/	25	/	25	,
14	/	22	/	26	,
2	/	26	/	26	,
26	/	2	/	26	,
22	/	14	/	26	,
22	/	14	/	26	,
2	/	27	/	27	,
2	/	28	/	28	,
28	/	2	/	28	,
22	/	19	/	29	,
2	/	29	/	29	,
26	/	13	/	29	,
2	/	30	/	30	,
30	/	2	/	30	,
2	/	31	/	31	,
26	/	17	/	31	,
2	/	32	/	32	,
32	/	2	/	32	,
2	/	33	/	33	,
26	/	22	/	34	,
22	/	26	/	34	,
2	/	34	/	34	,
34	/	2	/	34	,
2	/	35	/	35	,
20	/	30	/	36	,
12	/	34	/	36	,
2	/	36	/	36	,
36	/	2	/	36	,
34	/	12	/	36	,
30	/	20	/	36	,
34	/	12	/	36	,
2	/	37	/	37	,
2	/	38	/	38	,
38	/	2	/	38	,
30	/	25	/	39	,
2	/	39	/	39	,
38	/	9	/	39	,
2	/	40	/	40	,
40	/	2	/	40	,
2	/	41	/	41	,
34	/	23	/	41	,
18	/	38	/	42	,
2	/	42	/	42	,
42	/	2	/	42	,
38	/	18	/	42	,
22	/	37	/	43	,
2	/	43	/	43	,
34	/	28	/	44	,
28	/	34	/	44	,
2	/	44	/	44	,
44	/	2	/	44	,
2	/	45	/	45	,
26	/	38	/	46	,
2	/	46	/	46	,
46	/	2	/	46	,
38	/	26	/	46	,
2	/	47	/	47	,
2	/	48	/	48	,
48	/	2	/	48	,
38	/	31	/	49	,
14	/	47	/	49	,
2	/	49	/	49	,
46	/	17	/	49	
}
\draw 
          (.5*\a / \c,  .5* \b / \c  -.5)  circle  (1/\c) 
          (.5*\a / \c, -.5*\b / \c -.5)  circle  (1/\c)
          (-.5*\a / \c, .5*\b / \c - .5)  circle  (1/\c)
          (-.5*\a / \c, -.5*\b / \c -.5)  circle  (1/\c) 
          (.5*\a / \c +1,   .5* \b / \c  -.5)  circle  (1/\c) 
          (.5*\a / \c  +1, -.5*\b / \c -.5)  circle  (1/\c)
          (-.5*\a / \c +1, .5*\b / \c - .5)  circle  (1/\c)
          (-.5*\a / \c +1, -.5*\b / \c -.5)  circle  (1/\c) 
          (.5*\a / \c  -1,   .5* \b / \c  -.5)  circle  (1/\c) 
          (.5*\a / \c  -1, -.5*\b / \c -.5)  circle  (1/\c)
          (-.5*\a / \c  -1, .5*\b / \c - .5)  circle  (1/\c)
          (-.5*\a / \c  -1, -.5*\b / \c -.5)  circle  (1/\c) ;


\foreach \a/\b/\c   in {  
0/0/	1,  1/0/2,  0/2/3,  3/4/6, 0/4/15,
8/6/11,  5/12/14,   15/8/18,  8/12/23
}
\draw [color=red]
          (.5*\b/ \c, .5*\a/ \c  -.5)  circle  (.5/\c)
          (.5*\b/ \c, -.5*\a/ \c  -.5)  circle  (.5/\c)
          (-.5*\b/ \c, .5*\a/ \c  -.5)  circle  (.5/\c)
          (-.5*\b/ \c, -.5*\a/ \c -.5)  circle  (.5/\c) ;
\draw [color=red] (-3,0)--(3,0);
\draw [color=red] (-3,-1)--(3,-1);

\foreach \a/\b/\c in {
0/0/1,  4/3/4,  12/5/12, 6/8/9,  12/8/9,
4/-3/4,  12/-5/12,  6/-8/9,  12/-8/9
}
\draw [color=red, thick] (.5*\a/\c, .5*\b/\c -.5) circle (.5/\c)         (-.5*\a/\c, .5*\b/\c -.5) circle (.5/\c)
          (.5*\a/\c+1, .5*\b/\c -.5) circle (.5/\c)     (-.5*\a/\c -1, .5*\b/\c -.5) circle (.5/\c)
          (.5*\a/\c+2, .5*\b/\c -.5) circle (.5/\c)     (-.5*\a/\c -2, .5*\b/\c -.5) circle (.5/\c);

\end{tikzpicture}
\caption{The Super-Dedekind tessellation, a fragment.}
\label{fig:super}
\end{figure}
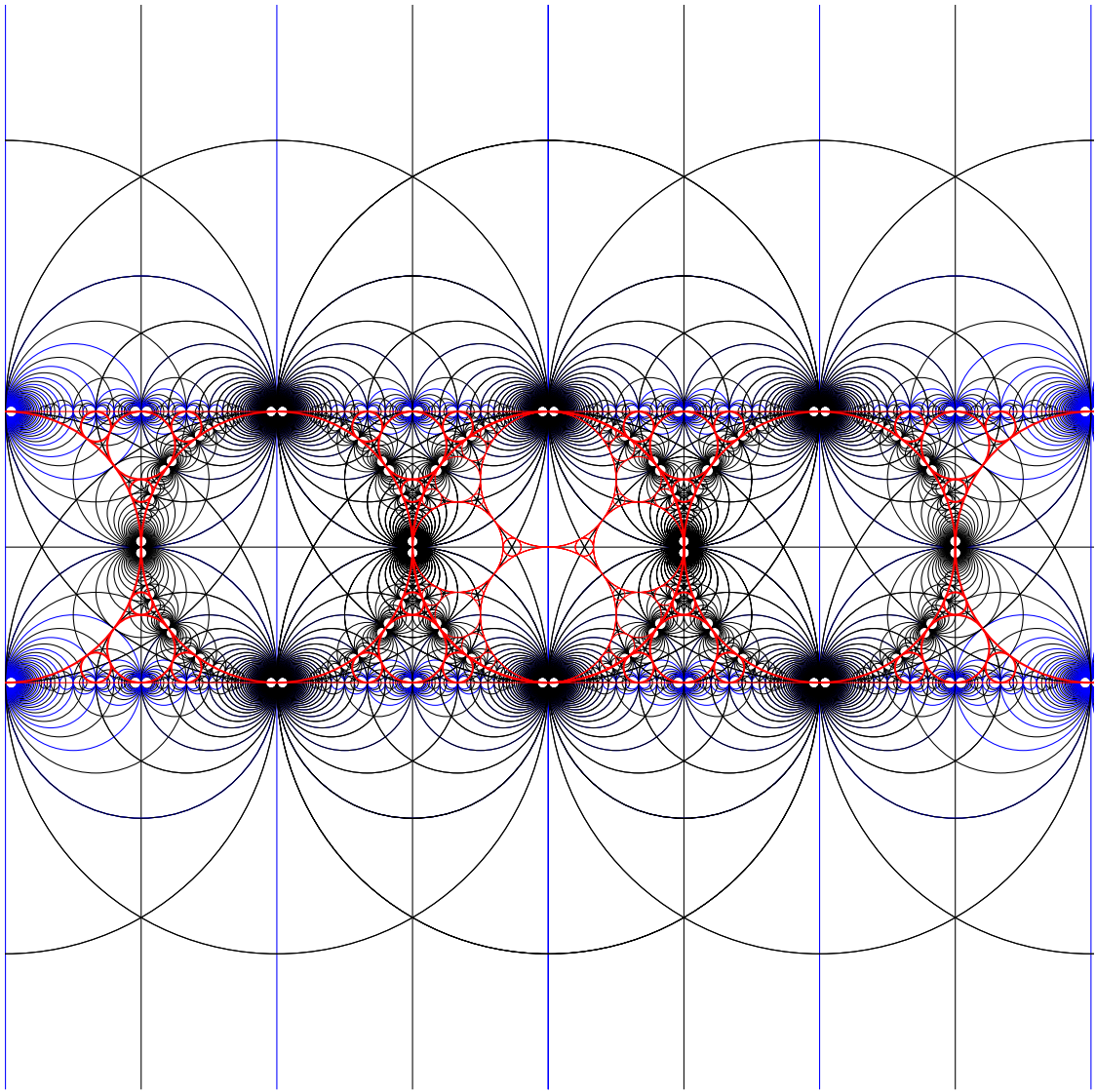

\newpage
\noindent
{\bf A note on M\"obius maps}
\\\\
The transformations \eqref{eq:FSR} and \eqref{eq:H}
may be expressed in terms of linear fractional maps:  
\begin{equation}
\label{eq:geo}
~\hspace{-.77in}
{\hbox{\sf Geometric}\atop\hbox{\sf version}}
\qquad
\begin{array}{rl}
T: z \mapsto z +1 \ &= \   
\begin{bmatrix}
                   1 & 1 \\ 
                   0 & 1
\end{bmatrix} 
 \cdot  z
 \\[11pt]
F: z \mapsto \bar z + i \ &= \   
\begin{bmatrix}
                   1 & i \\ 
                   0 & 1
\end{bmatrix} 
 \cdot \bar z
\\[11pt]
S: z \mapsto \dfrac{-1}{\bar z} \ &= \   
\begin{bmatrix}
                   0 & \!\!-1 \\ 
                   1 &  0
\end{bmatrix} 
 \cdot \bar z
\\[11pt]
R: z \mapsto \dfrac{i\bar z}{\bar z +i} \ &= \   
\begin{bmatrix}
                   1 & 0 \\ 
                   -i &  1
\end{bmatrix} 
 \cdot \bar z
\\[11pt]
H: z \mapsto -\bar z \ &= \   
\begin{bmatrix}
                   i & 0 \\ 
                   0 &  -i
\end{bmatrix} 
 \cdot \bar z
 \end{array}
 \end{equation}

Note that for consitency we should replace the translation $T$
by reflection through the vertical line $=1/2$, i.e. by 
$$
P:  z\ \mapsto \ 1-\bar z \ =\  \begin{bmatrix}
                   \! -i & i \\ 
                   0 & i
\end{bmatrix} 
 \cdot \bar z
$$
Translation is now recovered as $T = PH$, and $T^{-1} = HP$.  

Algebraic investigations may benefit from replacing the above 
with a subgroup of ${\rm PSL}(2,\mathbb Z[i])$ 
with the following alternatives that omit the use of the conjugation of $z$:
\begin{equation}
\label{eq:alg}
~\hspace{-.77in}
{\hbox{\sf Algebraic }\atop\hbox{\sf version}}
\qquad
\begin{array}{rl}
\hat T: z \mapsto  z +1 \ &= \   
\begin{bmatrix}
                   1 & 1 \\ 
                   0 & 1
\end{bmatrix} 
 \cdot  z
\\[11pt]
\hat F: z \mapsto -z + i \ &= \   
\begin{bmatrix}
                   i & 1 \\ 
                   0 & -i
\end{bmatrix} 
 \cdot  z
\\[11pt]
\hat S: z \mapsto \dfrac{-1}{z} \ &= \   
\begin{bmatrix}
                   0 & -1 \\ 
                   1 &  0
\end{bmatrix} 
 \cdot  z
\\[11pt]
\hat R: z \mapsto \dfrac{i z}{ z -i} \ &= \   
\begin{bmatrix}
                   i & 0 \\ 
                   1 & -i
\end{bmatrix} 
 \cdot  z
\\[11pt]
 \end{array}
 \end{equation}
Note that the transformations $\hat T$ and $\hat S$ generate 
the modular group ${\rm SL}(2,\mathbb Z)$.
We excluded an equivalent of $H$, which would have to be  
$\hat H z = -z$, but this would allow for $FH)^n\cdot z$ which moves $z$ in the complex plane 
vertically, $z\mapsto z+ni$, which cleaely does not belong to the symmetry groups
of the pattern.

\newpage

\section{Addenda}

\subsection{Game, the algorithm redefined in terms of $z$}

The process $\pi$ described in \eqref{eq:abc} for $\mathbb R^3$ 
may be redefined for the projective spinor space $\hat{\mathbb C}$
as a chain of maps that consistently push some initial $z$ towards the ``dark'' region of $\delta=0$.
We may facilitate it using two maps $\alpha,\beta:\mathbb C \to \mathbb C$.
The first is the inverse of $T$:
$$
\alpha: z\to z-1
$$
Map $\beta$ is a combination of reflections $S$ with conditional reflection $F$, 
and may be written as a single algebraic expression with the use of the  absolute values,
as shown in the box below.

Initial value $z$ can be chosen (or adjusted) to be on the right side of the plane, ${\rm Re} z>0$.
The process is contained in 
the upper part of the color belt,
$1/2 \geq y \geq 1$.
The goal is to bring $z$ to the ``dark disk'' $|z-1/2|<1/4$.
Here is the algorithm spell out:  

\begin{figure}[H]
\begin{minipage}[c]{.2\textwidth}
~~~
\end{minipage}
\boxed{\qquad
\begin{minipage}[c]{.5\textwidth}
\small
~\\
{\bf Algorithm redefined for $z\in C$}
~\\\\
0.  INPUT (x,y) \\
\phantom{.}\quad  $x=|x|$,  \ $d=0$ \\
1.  $ \omega = x^2+(1/2-y)^2$\\
2.  IF $\omega$ < 1/4, RETURN $d$, END  \\
\phantom{.}~\qquad     else\\ 
\phantom{.}~\qquad         if $x^2+(y-1)^2 < 1$ then \\
\phantom{.}~\qquad\qquad             $ x \ \to \ \frac{x}{x^2+(y-1)^2}$ \\
\phantom{.}~\qquad\qquad             $ y \ \to \  \frac{1}{2}+\left|\frac{x^2+y^2-y}{x^2+(y-1)^2}-\frac{1}{2}\right|$ \\
\phantom{.}~\qquad       end if\\
\phantom{.}~\qquad    $x := |x-1|$\\ 
\phantom{.}~\qquad    $d  := d+1$ \\
3. GO TO 1.\\
\end{minipage}
}
\caption{The algorithm for the motion in the projective tangency Pauli spinor space
that mimics that of the process of finding the depth of a tricycle}
\end{figure}

The string of operations is thus a word in the alphabet $\{\alpha,\beta\}$, starting with $\alpha$.
Since $\beta^2=\id$, the word is of the form
$$
\alpha^{n_1}\beta\alpha^{n_2}\beta\alpha^{n_3}B\; \ldots \; \beta\alpha^{n_m}
$$
for some $m$, and
$$
\delta = \sum_i n_i\,.
$$

\newpage
\subsection{ Remark on ambiguity of spinor description}

Two tangency spinors of Figure \ref{fig:a+b} totally determine the sizes of the tricycle 
and orientation in the Euclidean space 
(only the position is translationally not determined).
However, in trying to establish the map from the space of tricycles 
to the couples of tangency spinors (or Pauli spinors):
$$
\hbox{tricycles (up to translation)}\  \to \ \mathbb C^2
$$
we encounter an ambiguity.
The problem is in the choice of which disk is the anchor fort the spins $\mathbf a$ and $\mathbf b$ 
in Figure \ref{fig:a+b}.
We will tackle this ambiguity now.
Figure \ref{fig:6spinors} shows a more regular notation for all possible spinors.
With the notation shown we have 
$$
\begin{array}{lll} 
\mathbf a= \tspin(B,C)
&\qquad
&\mathbf a^\str= \tspin(C,B)  = \  i \mathbf a
\\
\mathbf b= \tspin(C,A)
&&\mathbf b^\str= \tspin(A,C)  \ = \  i \mathbf b
\\
\mathbf c= \tspin(A,B)
&&\mathbf c^\str  = \tspin(B,A) \ = \  i \mathbf c
\\[7pt]
\mathbf a + \mathbf b + \mathbf c= 0
&&
\mathbf a^\str +  \mathbf b^\str + \mathbf c^\str =  0
\end{array}
$$
The convention of the counterclockwise order of spinors for cross product holds:
$$
\mathbf b \times \mathbf a^\str = C
\,, \qquad
\mathbf c \times \mathbf b^\str = A
\,, \qquad
\mathbf a \times \mathbf c^\str = B
\,.
$$

\begin{figure}[t]
\centering
\def\rc{1}
\def\ra{.7}
\def\rb{.45}   
\def\xb{1.079}
\def\yb{.975}
\begin{tikzpicture}[scale=4.0, rotate=0]
\clip (-.2,-.3) rectangle (1.7,1.05);
\draw [fill=gold!10, thick] (0, 0) circle (\rc);
\draw [fill=gold!10, thick] (\rc+\ra,0) circle (\ra);
\draw [fill=gold!10, thick] (\xb, \yb) circle (\rb);

\node [scale=1.6] at (0, -\rc/7) {$\sf C$};
\node [scale=1.6] at (\rc+.57,   -\rc/7) {$\sf A$};
\node [scale=1.6] at (\xb +.16,  \yb-.02) {$\sf B$};
%
\draw [red,->, line width=1.2mm] (\rc+.03-.21, 0.1)--(\rc+.03+.21, 0.1);
\node  [scale=1.4]  at (\rc-.2, 0+.21)  {$\mathbf b$};
\draw [green,<-, line width=1.2mm] (\rc-.2, -.07)--(\rc+.2, -.07);
\node  [scale=1.4]  at (\rc+.3, 0-.14)  {$\mathbf b^\str$};

\draw [red,<-, line width=1.2mm] (.77-.15, .62-.4/3)--(.77+.15, .62+.4/3);   
\node [scale=1.4]  at (.98, .68)  {$\mathbf a$};
\draw [green,->, line width=1.2mm] (.7-.15, .75-.4/3)--(.7+.15, .75+.4/3);   
\node [scale=1.4]  at (.47, .71)  {$\mathbf a^\str$};

\draw [red,->, line width=1.2mm] (1.26+.08, .56-.15)--(1.26-.08, .56+.15);   
\node [scale=1.4]  at (.\rc+1.12, .41)  {$\mathbf c$};
\draw [green,<-, line width=1.2mm] (1.41+.08, .6-.15)--(1.41-.08, .6+.15);   
\node [scale=1.4]  at (.\rc+1.32, .82)  {$\mathbf c^+$};

\end{tikzpicture}
\caption{Tricycle and two spinors}
\label{fig:6spinors}
\end{figure}
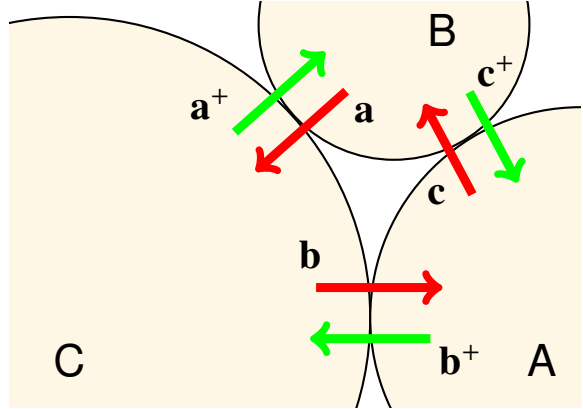

The complex representation is here more convenient because of the simple definition of the 
symplectic conjugation $+$.  
Recall that $\mathbf (a^+)^+ = -\mathbf a$.
Also, note that changing the sign in all spinors will not affect the signs of the curvatures
calculated by the cross-products.
The projection \eqref{eq:project} should be now written as  
$$
\mathbf b, \; \mathbf a^\str \quad \mapsto \quad 
\begin{bmatrix}
                   \mathbf b\  \\ \mathbf  a^\str
                   \end{bmatrix} 
\quad  \mapsto \quad z=\frac{\mathbf b\ }{\mathbf a^\str}\,,
$$
and similarly for the other two pairs.
We shall now show that the different  choices of anchors lead to 
complex number representation that differ by an equivalency defined by a subgroup of $\Theta$.
\def\pp{\phantom{{}^\str}}    
$$
\begin{array}{rlll}
z_1\ &= \ \dfrac{\mathbf b\pp}{\mathbf a^\str} \ = \ z 
\\[14pt]
z_2 \ &= \ \dfrac{\mathbf a\pp}{\mathbf c^\str} \ = \ 
         \dfrac{\mathbf a}{-\mathbf a^\str-\mathbf b^\str} \ = \ 
         \dfrac{1}{-i+\frac{\mathbf b\pp}{\mathbf a^\str}} \ = \ 
         \dfrac{-1}{i - z} \ = \ 
         SF\, z
\\[17pt]
z_3 \ &= \ \dfrac{\mathbf c\pp}{\mathbf b^\str} \ = \ 
         \dfrac{-\mathbf a-\mathbf b}{~~\mathbf b^\str} \ = \ 
         -\dfrac{\mathbf a}{i\mathbf b} - \dfrac{\mathbf b}{i\mathbf b} \ = \ 
         \dfrac{i\mathbf a}{\mathbf b} +i \ = \ 
         i - \dfrac{-1}{ z} \ = \ 
         FS\, z \ = \ FS\, z
\end{array}
$$

This defines a three-element group of $\Theta_0$:
$$
G\ = \ \{\id,\, FS,\, SF\} \ \cong \ \mathbb Z_3
$$
Note that we can use algebraic version of the transformations, 
i.e., $FS= \hat F\hat S$ and $SF=\hat S\hat F$. 
(An even number of transformations in \eqref{eq:geo} 
will ``cancel'' the complex conjugations).
The action of this group interchanges the regions
$$
\begin{array}{rll}
X&=\{\;z\;\big|\; \hbox{Im}\, z>1/2 \:|\;\hbox{and}|\; |z|\geq 1\,\} \\[7pt]
Y&=\{\;z\;\big|\; \;|z|\leq 1\;\hbox{and}\; |z-1/2|\leq 1\,\} \\[7pt]
Z&=\{\;z\;\big|\; \hbox{Im} \,z>0| \;\hbox{and}\; |z|\leq 1/2\,\}\,,
\end{array}
$$
any of which can be chosen as the fundamental domain.
Under the identification defined by the action of this group, we
obtain the space topologically homeomorphic to sphere, $S^2$,
which is easy to see when we pick the cigar-like region $Y$
and identify the top and bottom edges. 
The whole spinor complex plane $\mathbb C\cup \{\infty\}$,
which topologically is also a sphere, ``wraps'' around the sphere $Y$ three times,
forming a fiber bundle with discrete fibers coinciding with the orbits of $G$
in $\dot{\mathbb C}$.

\newpage
\subsection{ Remark on integrality}

\begin{proposition}
\label{thm:rational}
Every rational disk packing is integral by scaling by some integer.
\end{proposition}

\begin{proof}  Pick a quadruple of disks in the packing in the Descartes configuration.
As a finite set, it may be scaled to an integral quadruple.
Hence all discs are integral.
\end{proof}

\begin{proposition}
Every rational Pauli spinor generates a (scaled) integral packing.
\end{proposition}

\begin{proof} 
Equation \eqref{eq:ABC} implies that the disk $D$ that completes the tricycle
determined by $z$ 
to the Descartes configuration is also rational.
By Proposition \ref{thm:rational}, the claim holds.
\end{proof}

\begin{proposition}
Only rational x,y generates (scaled) integral disk packing.
\end{proposition}

\begin{proof}
Referring to parametrization \eqref{eq:ABCxy}, since $C$ is rational, so is $y$ by (i). 
Since $B=x^2+y^2 -y$ is rational, so is $x^2$.  
Thus $x=\sqrt{q}$ for some $q\in \mathbb Q$.
Now, since $D$ is rational, 
$1+\sqrt{q}^2 \in \mathbb Q$  (by (iv).
Thus 
$\sqrt{q}\in \mathbb Q$,
or $q$ is a square.
\end{proof}

Hence, the integral packings correspond to (are coded by) the rational complex numbers , $z\in \mathbb Q[i]$,
forming a dense subset of the plane $\mathbb C$.
A related picture may be found in \cite{jk-lattice}.

%
%
%

\section*{Appendix A: Apollonian disk packings}

An {\bf Apollonian disk packing} is an arrangement of an infinite number of disks.
Such an arrangement may be constructed by starting with a tricycle, 
called in this context a {\bf seed},
and completing recursively every tricycle already constructed to a Descartes configuration.

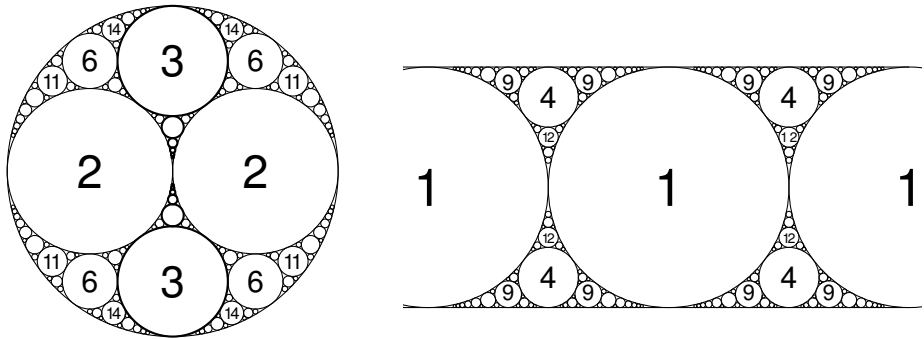
\begin{figure}[H]
\centering
\begin{tikzpicture}[scale=2.2]
\draw (0,0) circle (1);
\foreach \a/\b/\c   in {
1 / 0 / 2 
}
\draw (\a/\c,\b/\c) circle (1/\c)
          (-\a/\c,\b/\c) circle (1/\c);
\foreach \a/\b/\c in {
0 / 2 / 3 ,
0 /4 /15 ,
0 / 6 / 35 , 
0 / 8/ 63,
0 /10 / 99,
0 / 12 / 143
}
\draw[thick] (\a/\c,\b/\c) circle (1/\c)
          (\a/\c,-\b/\c) circle (1/\c) ;
\foreach \a/\b/\c/\d in {
3 / 4 /6, 	8 / 6 / 11,	5 / 12/ 14,	15/ 8 / 18,	8 / 12 / 23,	7 / 24 / 26,
24/	10/	27, 	21/	20/	30, 	16/	30/	35, 	3/	12/	38, 	35/	12/	38, 	24/	20/	39, 	9/	40/	42,
16/	36/	47, 	15/	24/	50, 	 48/	14/	51, 	45/	28/	54, 	24/	30/	59, 	40/	42/	59, 	11/	60/	62,
21/	36/	62, 	48/	28/	63, 	33/	56/	66, 	63/	16/	66, 	8/	24/	71, 	55/	48/	74, 	24/	70/	75,
48/	42/	83, 	80/	18/	83, 	13/	84/	86, 	77/	36/	86, 	24/	76/	87, 	24/	40/	87, 	39/	80/	90
}
\draw (\a/\c,\b/\c) circle (1/\c)       (-\a/\c,\b/\c) circle (1/\c)
          (\a/\c,-\b/\c) circle (1/\c)       (-\a/\c,-\b/\c) circle (1/\c) 
;
\node at (-1/2,0) [scale=1.7, color=black] {\sf 2};
\node at (1/2,0) [scale=1.7, color=black] {\sf 2};
\node at (0,2/3) [scale=1.6, color=black] {\sf 3};
\node at (0,-2/3) [scale=1.6, color=black] {\sf 3};
\node at (1/2,2/3) [scale=1.1, color=black] {\sf 6};
\node at (-1/2,2/3) [scale=1.1, color=black] {\sf 6};
\node at (1/2,-2/3) [scale=1.1, color=black] {\sf 6};
\node at (-1/2,-2/3) [scale=1.1, color=black] {\sf 6};
\node at (8/11,6/11) [scale=.77, color=black] {\sf 1$\!$1};
\node at (-8/11,6/11) [scale=.77, color=black] {\sf 1\!1};
\node at (8/11,-6/11) [scale=.77, color=black] {\sf 1\!1};
\node at (-8/11,-6/11) [scale=.77, color=black] {\sf 1\!1};
\node at (5/14,6/7) [scale=.6, color=black] {\sf 1\!4};
\node at (-5/14,6/7) [scale=.6, color=black] {\sf 1\!4};
\node at (5/14,-6/7) [scale=.6, color=black] {\sf 1\!4};
\node at (-5/14,-6/7) [scale=.6, color=black] {\sf 1\!4};
\end{tikzpicture}
\qquad
%
\begin{tikzpicture}[scale=1.6, rotate=90, shift={(0,2cm)}]  
\clip (-1.25,-2.1) rectangle (1.1,2.2);
\draw (1,-2) -- (1,3);
\draw (-1,-2) -- (-1,3);
\draw (0,0) circle (1);
\draw (0,2) circle (1);
\draw (0,-2) circle (1);
\foreach \a/\b/\c in {
3/4/4,  5/12/12,  7/24/24, 9/40/40  
}
\draw (\a/\c,\b/\c) circle (1/\c)    (\a/\c,-\b/\c) circle (1/\c)
          (-\a/\c,\b/\c) circle (1/\c)    (-\a/\c,-\b/\c) circle (1/\c)   ;
\foreach \a/\b/\c in {
8/    6/   9, 	
15/   8/   16 , 	
24/  20/  25, 	
24/  10/  25, 	
21/  20/  28,
16/	30/	33,    
35/	12/	36,
48/	42/	49,
48/	28/	49,
48/	14/	49,
45/	28/	52,
40/	42/	57,
33/	56/	64,
63/	48/	64,
63/	16/	64,
55/	48/	72,
24/	70/	73,
69/	60/	76,
80/	72/	81,
64/	60/	81,
80/	36/	81,
80/	18/	81
}
\draw (\a/\c, \b/\c) circle (1/\c)          (-\a/\c, \b/\c) circle (1/\c)
          (\a/\c,-\b/\c) circle (1/\c)         (-\a/\c,-\b/\c) circle (1/\c)
          (\a/\c,2-\b/\c) circle (1/\c)       (-\a/\c,2-\b/\c) circle (1/\c)
          (\a/\c,2+\b/\c) circle (1/\c)       (-\a/\c,2+\b/\c) circle (1/\c)
          (\a/\c,-2+\b/\c) circle (1/\c)       (-\a/\c,-2+\b/\c) circle (1/\c)
;
\node at (0,0) [scale=1.9, color=black] {\sf 1};
\node at (0,-2) [scale=1.9, color=black] {\sf 1};
\node at (0,2) [scale=1.9, color=black] {\sf 1};
\node at (-3/4,1) [scale=1.1, color=black] {\sf 4};
\node at (-3/4,-1) [scale=1.1, color=black] {\sf 4};
\node at (3/4,1) [scale=1.1, color=black] {\sf 4};
\node at (3/4,-1) [scale=1.1, color=black] {\sf 4};
\node at (7/8,8/6) [scale=.8, color=black] {\sf 9};
\node at (7/8,-8/6) [scale=.8, color=black] {\sf 9};
\node at (7/8,4/6) [scale=.8, color=black] {\sf 9};
\node at (7/8,-4/6) [scale=.8, color=black] {\sf 9};
\node at (-7/8,8/6) [scale=.8, color=black] {\sf 9};
\node at (-7/8,-8/6) [scale=.8, color=black] {\sf 9};
\node at (-7/8,4/6) [scale=.8, color=black] {\sf 9};
\node at (-7/8,-4/6) [scale=.8, color=black] {\sf 9};
\node at (-0.42,1) [scale=.5, color=black] {\sf 1\!2};
\node at (-0.42,-1) [scale=.5, color=black] {\sf 1\!2};
\node at (.42,1) [scale=.5, color=black] {\sf 1\!2};
\node at (.42,-1) [scale=.5, color=black] {\sf 1\! 2};
\end{tikzpicture}
\caption{Apollonian Window (left) and Apollonian Belt (right)}
\label{fig:ApollonianPackings}
\end{figure}

\newpage
\section*{Appendix B:  The graph of the tricycles in an Apollonian packing}

The process that defines the value of the depth of a tricycle
is just one of many paths in the (discrete) space of tricycles in a given Apollonian disk packing.
To see it in he context, let us define a graph in t: 
The vertices correspond to tricycles. 
Two tricycles are joined by an edge 
if they share two disks and the non-shared disks are tangent to each other in the packing.
The process $\pi$ makes a continuous (in the sense of the graph topology) path in this graph.
Figure \ref{fig:graph} presents this graph.
As shown, It consist of an infinite number of tetrahedrons connected pairwise by the vertices.
The vertices are 6-valent: this is obvious, 
for every tricycle may acquire anew disk in two ways, and lose one in three ways.b

\begin{figure}[H]
\centering
\includegraphics[scale=.9]{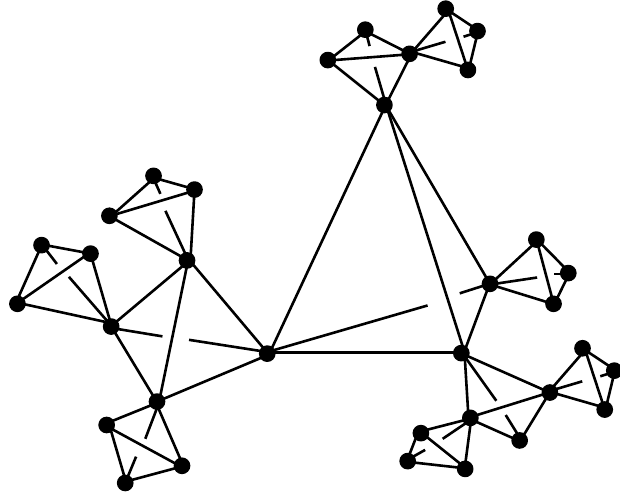}
\qquad
\includegraphics[scale=.6]{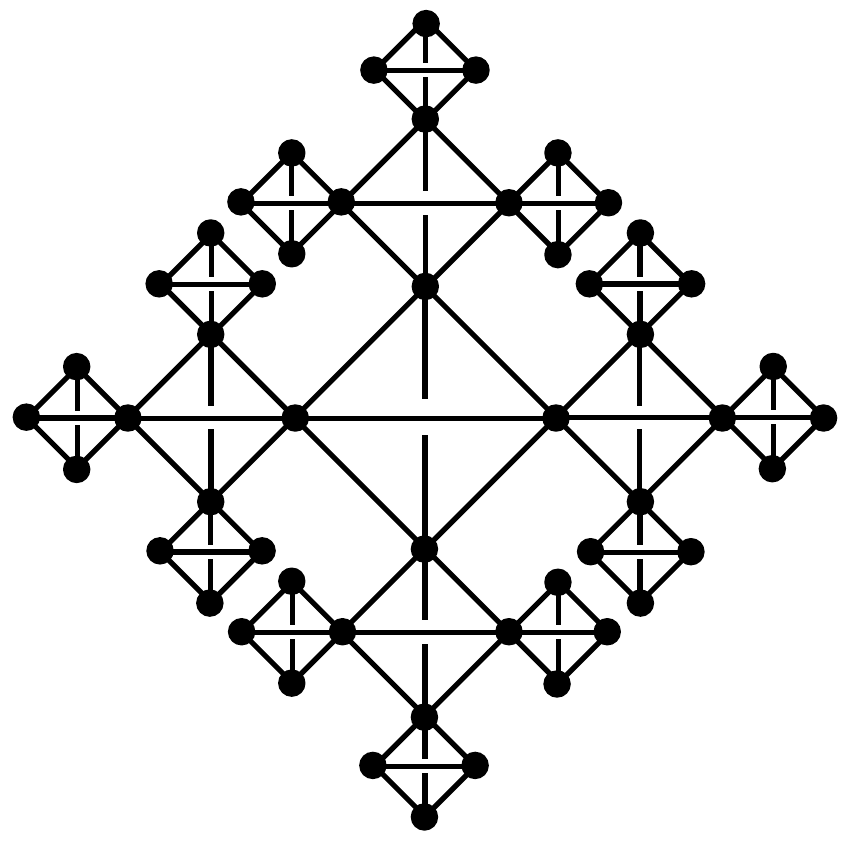}
\caption{The 3-cycle graph in an Apollonian gasket.  Left:  3-dimensionality emphasized.  Right: 2D version}
\label{fig:graph}
\end{figure}

The sub-graph of the tricycles that have value $\delta=0$ are the ones that 
contain the disks of non-positive curvature.
Apollonian packing can have only one negative disk,
and the the tricycles that contain it must have the other two disk in its corona.
Figure \ref{fig:subgraph} shows this subgraph.
The process seeks the shortest path to this subgraph starting from a given vertex.
 
\begin{figure}[H]
\centering
\includegraphics[scale=.57]{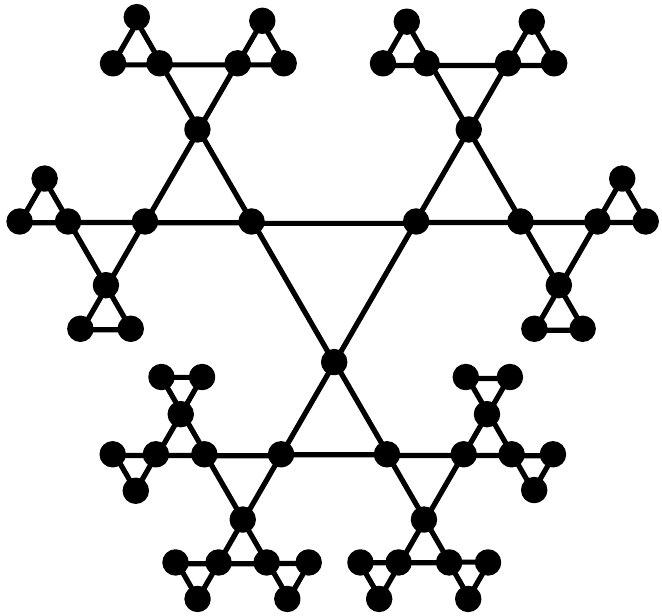}
\caption{The 3-cycle graph in an Apollonian gasket.  Left:  3-dimensionality emphasized.  Right: 2D version}
\label{fig:subgraph}
\end{figure}


\end{document}